\documentclass[12pt]{amsart}
\usepackage{amssymb}
\usepackage{amsmath}
\usepackage{amsthm}
\usepackage{graphicx}
\usepackage{enumerate}
\usepackage{tikz}
\usepackage{tabularx}
\usepackage{tabmac_avi}

\newtheorem{thm}{Theorem}[section]
\newtheorem{prop}[thm]{Proposition}
\newtheorem{cor}[thm]{Corollary}
\newtheorem{lem}[thm]{Lemma}

\newtheorem{defn}[thm]{Definition}

\newtheorem{prob}[thm]{Problem}

\numberwithin{equation}{section}

\newcommand{\sn}{\mathfrak{S}_n}
\newcommand{\sr}{\mathfrak{S}_r}
\newcommand{\mfs}[1]{\mathfrak{S}_{#1}}
\newcommand{\mfsmin}[1]{\mathfrak{S}_-^{#1}}
\newcommand{\mfsrmin}[1]{\mathfrak{S}_{#1}^-}
\newcommand{\zsn}{\mathbb{Z}[\sn]}

\newcommand{\zx}{\mathbb{Z}[x]}

\newcommand{\zxn}{\mathbb{Z}[x_{1,1},\dotsc,x_{n,n}]}
\newcommand{\cqq}{\mathbb{C}[\qp12, \qm12]}
\newcommand{\zqq}{\mathbb{Z}[\qp12, \qm12]}

\newcommand{\csn}{\mathbb{C}[\sn]}

\newcommand{\eJ}{{\emptyset, J}}

\newcommand{\wo}{w_0}

\newcommand{\euv}{\epsilon_{u,v}}

\newcommand{\evw}{\epsilon_{v,w}}

\newcommand{\qp}[2]{q^{\frac{#1}{#2}}}
\newcommand{\qm}[2]{q^{\negthinspace\Bar\,\frac{#1}{#2}}}
\newcommand{\qumongous}[4]{q^{\frac{#1}2 + #2 + #3 - \frac{#4}2}}

\newcommand{\quv}{q_{u,v}}

\newcommand{\qvw}{q_{v,w}}

\newcommand{\qeu}{\smash{\qp{\ell(u)}2}}
\newcommand{\qustar}{\smash{\qp{\ell(u^*)}2}}
\newcommand{\qet}{\smash{\qp{\ell(t)}2}}

\newcommand{\qev}{\smash{\qp{\ell(v)}2}}
\newcommand{\qew}{\smash{\qp{\ell(w)}2}}
\newcommand{\qwstar}{\smash{\qp{\ell(w^*)}2}}
\newcommand{\qwm}{\smash{\qp{\ell(w_-)}2}}
\newcommand{\qwmt}{\smash{\qp{\ell(w_-t)}2}}
\newcommand{\qwmtstar}{\smash{\qp{\ell((w_-t)^*)}2}}

\newcommand{\qiuv}{q_{u,v}^{-1}}

\newcommand{\qivw}{q_{v,w}^{-1}}

\newcommand{\qiew}{\qm{\ell(w)}2}
\newcommand{\qiey}{\qm{\ell(y)}2}
\newcommand{\qiev}{q_v^{-1}}

\newcommand{\qdiff}{\qp12 - \qm12}

\newcommand{\wT}{\widetilde T} 
\newcommand{\wTp}{\widetilde T'}
\newcommand{\ol}[1]{\overline{#1}}
\newcommand{\wR}[1]{\widetilde R_{#1}}
\newcommand{\wS}[1]{\widetilde S_{#1}}

\newcommand{\hnq}{H_n(q)}

\newcommand{\A}{\mathcal{A}}
\newcommand{\anq}{\mathcal{A}(n,q)}

\newcommand{\arnq}{\mathcal{A}_r}

\newcommand{\Ann}{\mathcal{A}_{[n],[n]}}
\newcommand{\annnq}{\mathcal{A}_{[n],[n]}}
\newcommand{\anmnq}{\mathcal{A}_{[n],M}(n;q)}

\newcommand{\AnM}{\mathcal{A}_{[n],M}}
\newcommand{\ALM}{\mathcal{A}_{L,M}}

\newcommand{\Hpej}{H'_{\smash\eJ}} 
\newcommand{\Hej}{H_\eJ}

\newcommand{\Wejp}{W^{\emptyset,J}_+}

\newcommand{\Wijp}{W^{I,J}_+}
\newcommand{\Wijm}{W^{I,J}_-}

\newcommand{\pej}[1]{p_{#1}^{\emptyset,J}}

\newcommand{\rej}[1]{r_{#1}^{\emptyset,J}}

\newcommand{\redexp}[2]{s_{i_{#1}} \ntnsp \cdots s_{i_{#2}}}
\newcommand{\wtc}[2]{\widetilde{C}_{#1}(#2)}

\newcommand{\imm}[1]{\mathrm{Imm}_{#1}}

\newcommand{\sumsb}[1]{\sum_{\substack{#1}}}  
\newcommand{\inv}{\textsc{inv}}
\newcommand{\rinv}{\textsc{rinv}}
\newcommand{\defeq}{:=} 
\newcommand{\dfct}{\textsc{d}}
\newcommand{\dnc}{\textsc{dnc}}
\newcommand{\dc}{\textsc{dc}}

\newcommand{\spn}{\mathrm{span}}
\newcommand{\sgn}{\mathrm{sgn}}
\newcommand{\wgt}{\mathrm{wgt}}
\newcommand{\type}{\mathrm{type}}
\newcommand{\incross}{\textsc{invnc}}
\newcommand{\cdncross}{\textsc{cdnc}}
\newcommand{\cross}{\textsc{cr}}
\newcommand{\pavoiding}{$3412$-avoiding, $4231$-avoiding }
\newcommand{\avoidsp}{avoids the patterns $3412$ and $4231${}}
\newcommand{\avoidp}{avoid the patterns $3412$ and $4231${}}

\newcommand{\ul}[1]{\underline{#1}}  
\newcommand{\ssec}[1]{\subsection{#1}{$\negthinspace$}}

\newcommand{\tr}{{\negthickspace \top \negthickspace}}
\newcommand{\ntnsp}{\negthinspace}
\newcommand{\ntksp}{\negthickspace}
\newcommand{\nTksp}{\negthickspace\negthickspace}

\newcommand{\oqglnc}{{\mathcal O}_q(GL_n (\mathbb C))}

\newcommand{\oqslnc}{{\mathcal O}_q(SL_n (\mathbb C))}

\newcommand{\qdet}{\mathrm{det}_q}

\newcommand{\permmon}[2]{#1_{1,#2_1} \ntnsp\cdots {#1}_{n,#2_n}}
\newcommand{\sprod}[2]{s_{#1_1} \ntnsp\cdots s_{#1_#2}}

\newcommand{\ssm}{\smallsetminus}

\newcommand{\bs}{\backslash}  
\newcommand{\wleq}{\leq_W}
\newcommand{\wless}{<_W}  
\newcommand{\slambda}{\mathfrak{S}_\lambda}
\newcommand{\slambdamin}{\mathfrak{S}_\lambda^{-}}

\def\hhhpsp{\def\baselinestretch{0.15}\large\normalsize}
\def\ssp{\def\baselinestretch{1.0}\large\normalsize}

\begin{document}
\author{Adam Clearwater}
\author{Mark Skandera}
\title[Total nonnegativity and $\hnq$ characters]
      {Total nonnegativity and induced sign characters of the Hecke algebra}

\bibliographystyle{../dart}

\date{\today}

\begin{abstract}
  Let $\mfs{[i,j]}$ be the subgroup of the symmetric group $\sn$ generated by
  adjacent transpositions $(i,i+1), \dotsc, (j-1,j)$, assuming
  $1 \leq i < j \leq n$.
  We give a combinatorial rule for
  evaluating induced sign characters of the type $A$ Hecke algebra $\hnq$
  at all elements of the form $\sum_{w \in \mfs{[i,j]}} T_w$ and at all products
  of such elements. This includes evaluation at some elements $C'_w(q)$
  of the Kazhdan-Lusztig basis.
  
\end{abstract}

\maketitle

\section{Introduction}\label{s:intro}
Concepts of total nonnegativity, first explored by
Gantmacher and Krein~\cite{GantKreinOsc}
have found their way into many areas of mathematics,
including the study of polynomials $p(x_{1,1},\dotsc,x_{n,n})$
satisfying $p(a_{1,1}, \dotsc, a_{n,n}) \geq 0$ for every totally nonnegative
$n \times n$ matrix $A = (a_{i,j})$.
We call these {\em totally nonnegative (TNN)} polynomials.
In particular, work of Lusztig~\cite{LusztigTP} implies that
in $\zx \defeq \zxn$,
certain elements
which are related to the dual canonical basis of
the quantum group $\oqslnc$ are TNN polynomials.

In practice, it is sometimes possible to use cluster algebras~\cite{FominTPTest}
and a computer to demonstrate that a polynomial is TNN by expressing it
as a subtraction-free rational expression in matrix minors.
On the other hand, no simple characterization of TNN polynomials is known.
To improve our understanding of TNN polynomials,
one might begin by investigating the {\em immanant subspace}
$\spn_{\mathbb Z} \{ \permmon xw \,|\, w \in \sn \}$ of $\zx$,
especially the generating functions
\begin{equation}\label{eq:imm}
  \imm{\theta}(x) \defeq \sum_{w \in \sn} \theta(w) \permmon xw
\end{equation}
for class functions $\theta: \sn \rightarrow \mathbb Z$.
Or, since some published results relate total nonnegativity to
the Hecke algebra $\hnq$ and its traces $\theta_q: \hnq \rightarrow \zqq$,
one might investigate these.


In particular, let $\{ \wtc wq \,|\, w \in \sn \}$ be the (modified, signless)
{\em Kazhdan-Lusztig basis} of $\hnq$, defined by
\begin{equation}\label{eq:klbasis}
\wtc wq \defeq \qp{\ell(w)}2  C'_w(q) = \sum_{v \leq w} P_{v,w}(q) T_v,
\end{equation}
where $\{ T_w \,|\, w \in \sn \}$ is the natural basis of $\hnq$,
$\{ P_{v,w}(q) \,|\, v,w \in \sn \}$ are the
Kazhdan-Lusztig polynomials~\cite{KLRepCH}, and $\leq$ denotes the Bruhat order.
(See,\,e.g., \cite[\S 2]{BBCoxeter}.)
Specializing at $\qp12 = 1$ we have $T_v \mapsto v$ and $H_n(1) \cong \zsn$.
For $1 \leq a < b \leq n$, let $s_{[a,b]} \in \sn$ be the {\em reversal}
whose one-line notation is
$1 \cdots (a-1) b (b-1) \cdots (a+1) a (a+2) \cdots n$.
Stembridge~\cite[Thm.\,2.1]{StemImm} showed that
for any linear function $\theta: \zsn \rightarrow \mathbb Z$, the immanant
$\imm{\theta}(x)$ is TNN if for all sequences $J_1, \dotsc, J_m$ of
subintervals of $[n] \defeq \{1, \dotsc, n\}$, we have
\begin{equation}\label{eq:cprod1}
  \theta(\wtc{s_{J_1}\ntnsp}1 \cdots \wtc{s_{J_m}\ntnsp}1) \geq 0.
\end{equation}
Furthermore, by Lindstr\"om's Lemma~\cite{LinVRep}
and its converse~\cite{BrentiCTP}, \cite{CryerProp},
a combinatorial interpretation of the above
expression would immediately yield a combinatorial interpretation of
the number $\imm{\theta}(A)$ for $A$ a TNN matrix.
Now if a linear function
$\theta_q: \hnq \rightarrow \zqq$ specializes at $\qp12 = 1$ to $\theta$,
then (\ref{eq:cprod1}) is clearly a consequence of the condition
\begin{equation}\label{eq:cprodq}
  \theta_q(\wtc{s_{J_1}\ntnsp}q \cdots \wtc{s_{J_m}\ntnsp}q) \in \mathbb N[q],
\end{equation}
and a combinatorial interpretation of coefficients of the above polynomial
would yield combinatorial interpretations of the earlier expressions.
Haiman~\cite[Appendix]{HaimanHecke} observed that (\ref{eq:cprodq}) in turn
follows from the condition that for all $w \in \sn$ we have
\begin{equation}\label{eq:cnoprodq}
  \theta_q(\wtc wq) \in \mathbb N[q],
\end{equation}
since results in \cite{BBDFaisceaux}, \cite{SpringerQACI}
imply that any product of Kazhdan-Lusztig basis
elements belongs to $\spn_{\mathbb N[q]}\{ \wtc wq \,|\, w \in \sn \}$.
On the other hand, it is not clear that the coefficients in the resulting
linear combination would lead to a combinatorial interpretation of
the expression in (\ref{eq:cprodq}).

Stembridge~\cite[Cor.\,3.3]{StemImm}
proved that for each irreducible $\sn$-character
$\chi^\lambda$, the evaluation (\ref{eq:cprod1}) belongs to $\mathbb N$.
Haiman~\cite[Lem.\,1.1]{HaimanHecke}
proved that for each corresponding irreducible
$\hnq$-character $\chi_q^\lambda$, the evaluation (\ref{eq:cnoprodq})
belongs to $\mathbb N[q]$.
They conjectured the same~\cite[Conj.\,2.1]{StemConj},
\cite[Conj.\,2.1]{HaimanHecke}
for functions $\phi^\lambda$ ($\phi_q^\lambda$),
called {\em monomial traces}, related to irreducible characters by
the inverse Kostka numbers.
Both sets
$\{ \chi^\lambda \,|\, \lambda \vdash n \}$,
$\{ \phi^\lambda \,|\, \lambda \vdash n \}$
($\{ \chi_q^\lambda \,|\, \lambda \vdash n \}$,
$\{ \phi_q^\lambda \,|\, \lambda \vdash n \}$),
where $\lambda \vdash n$ denotes that $\lambda$ varies over all partitions
of $n$,
form bases of the space of $\sn$-class functions ($\hnq$-trace space).
(See, e.g., \cite[\S A3.9]{BBCoxeter}, \cite[\S 1]{StemConj}.)

None of these results or conjectures
included a combinatorial interpretation.
For the purposes of understanding TNN polynomials of the form (\ref{eq:imm}),
it would be desirable to solve the following problem.
\begin{prob}\label{prob:interp}
Give combinatorial interpretations of
all of the expressions in
(\ref{eq:cprodq})
when $\theta_q$ varies over all elements of any basis of the
$\hnq$-trace space.
\end{prob}
So far, only some special cases have such interpretations.
In the case that $w$ \avoidsp, the Kazhdan-Lusztig basis element
$\wtc wq$ is closely
related to a product of the form appearing in (\ref{eq:cprodq}).
Combinatorial interpretations of the corresponding expressions
(\ref{eq:cprodq}) and (\ref{eq:cnoprodq})
were given in \cite{CHSSkanEKL} for
$\theta_q \in \{\chi_q^\lambda \,|\, \lambda \vdash n\}$,
and for $\theta_q$ belonging to several other bases of the $\hnq$-trace space,
including the basis $\{ \epsilon_q^\lambda \,|\, \lambda \vdash n \}$
of induced sign characters.
Combinatorial interpretations for $\theta_q = \phi_q^\lambda$
were given only when
$\lambda$ has at most two parts~\cite[Thm.\,10.3]{CHSSkanEKL},
or when $\lambda$ has rectangular shape
and $q=1$~\cite[Thm.\,2.8]{StemConj}.

In the case that all reversals
$s_{J_1},\dotsc,s_{J_m}$
in (\ref{eq:cprodq})
are adjacent transpositions,
combinatorial interpretations of
\begin{equation*}
  \epsilon_q^\lambda(\wtc{s_{[j_1,j_1+1]}\ntnsp}q \cdots \wtc{s_{[j_m,j_m+1]}\ntnsp}q)
\end{equation*}
for all $\lambda \vdash n$ were given in \cite[Thm.\,5.4]{KLSBasesQMBIndSgn},
nearly solving Problem~\ref{prob:interp}.
To complete the solution of Problem~\ref{prob:interp},
we generalize \cite[Thm.\,5.4]{KLSBasesQMBIndSgn}
to interpret all evaluations of the form
\begin{equation}\label{eq:cprodegen}
  \epsilon_q^\lambda(\wtc{s_{J_1}\ntnsp}q \cdots \wtc{s_{J_m}\ntnsp}q),
\end{equation}
where the intervals $J_1, \dotsc, J_m \subseteq [n]$
have arbitrary cardinality. 
The justification of our solution closely follows the steps used in
\cite[\S 3-5]{KLSBasesQMBIndSgn},
with definitions and propositions generalized as necessary.

In Section~\ref{s:hnqanq}, we discuss the symmetric group algebra $\zsn$,
the Hecke algebra $\hnq$, and a noncommutative $q$-analog $\anq$ of $\zx$
known as the quantum matrix bialgebra.
The $q$-analogs $\imm{\theta_q}(x)$ of immanants (\ref{eq:imm})
belong to a $q$-analog $\annnq \subset \A$ of the immanant subspace of $\zx$,
and $\imm{\epsilon_q^\lambda}(x)$ has a natural expression
in terms of certain monomials which we write as
$\{x^{u,w} \,|\, u, w \in \sn \}$.
In Section~\ref{s:snets}, we consider
elements of the form
$\smash{\wtc{s_{J_1}\ntnsp}q \cdots \wtc{s_{J_m}\ntnsp}q} \in \hnq$,
and we associate to each a planar network $G$,
a related matrix $B$, and a linear map $\sigma_B: \Ann \rightarrow \zqq$.
In Section~\ref{s:gtabx},
we interpret $\sigma_B(x^{u,w}) \in \mathbb N[q]$ for each pair $u,w \in \sn$
in terms of statistics on families of source-to-sink paths in $G$
and on tableaux filled with such paths.
In Section~\ref{s:qimmevalthm}, we show that
for each $\hnq$-trace $\theta_q$ and its generating function
$\imm{\theta_q}(x)$, we have the identity
\begin{equation}\label{eq:evalwithsigma}
\theta_q(\wtc{s_{J_1}\ntnsp}q \cdots \wtc{s_{J_m}\ntnsp}q)
 = \sigma_B(\imm{\theta_q}(x)).
\end{equation}
This identity leads to our main result in Section~\ref{s:indsgneval},
where we use path families, tableaux, and
$\imm{\epsilon_q^\lambda}(x)$
to combinatorially evaluate the expressions (\ref{eq:cprodegen}). 


\section{Algebraic background}\label{s:hnqanq}

Define the {\em symmetric group algebra} $\zsn$
and the {\em (type $A$ Iwahori-) Hecke algebra} $\hnq$
to be the algebras with
multiplicative identity elements $e$ and $T_e$,
respectively,
generated over $\mathbb Z$ and $\zqq$
by elements
$s_1,\dotsc, s_{n-1}$ and 
$T_{s_1},\dotsc, T_{s_{n-1}}$, 
subject to the relations
\begin{equation}\label{eq:hnqdef}
\begin{alignedat}{3}
s_i^2 &= e &\qquad
T_{s_i}^2 &= (q-1) T_{s_i} + qT_e &\qquad
&\text{for $i = 1, \dotsc, n-1$},\\
s_is_js_i &= s_js_is_j &\qquad
T_{s_i}T_{s_j}T_{s_i} &= T_{s_j}T_{s_i}T_{s_j} &\qquad
&\text{for $|i - j| = 1$},\\
s_is_j &= s_js_i &\qquad
T_{s_i}T_{s_j} &= T_{s_j}T_{s_i} &\qquad
&\text{for $|i - j| \geq 2$}.
\end{alignedat}
\end{equation}
Analogous to the natural basis $\{ w \,|\, w \in \sn \}$ of $\zsn$
is the natural basis $\{ T_w \,|\, w \in \sn \}$ of $\hnq$,
where we define
$T_w = T_{s_{i_1}} \ntksp \cdots T_{s_{i_\ell}}$
whenever $\sprod i\ell$
is a reduced (short as possible)
expression for $w$ in $\sn$.  We call $\ell$ the {\em length}
of $w$ and write $\ell = \ell(w)$.
We define the {\em one-line notation} $w_1 \cdots w_n$ of $w \in \sn$ by
letting any expression for $w$ act on the word $1 \cdots n$,
where each generator $s_j = s_{[j,j+1]}$ acts on an $n$-letter word by
swapping the letters in positions $j$ and $j+1$, i.e.,
$s_j \circ v_1 \cdots v_n = v_1 \cdots v_{j-1} v_{j+1} v_j v_{j+2} \cdots v_n$.
$\ell(w)$ is equal
to $\inv(w)$, the number of inversions 
in the one-line notation $w_1 \cdots w_n$ of $w$.
If this one-line notation contains no subword matching the pattern
$3412$ or $4231$ (e.g., if $w$ is a reversal),
then the Kazhdan-Lusztig polynomials in (\ref{eq:klbasis})
are identically $1$~\cite{LakSan}.

For any subinterval $J = [a,b]$ of $[n] \defeq \{1, \dotsc, n\}$,
define $\mfs J$ to be the subgroup of $\sn$ generated by
$\{s_a, \dotsc, s_{b-1}\}$.
The longest element of $\mfs J \cong \mfs{b-a+1}$
is the reversal $s_J \defeq s_{[a,b]}$.
Define $\mfsmin J$
to be the set of minimum length representatives
of left cosets
$w \mfs J$,
i.e., the elements $w \in \sn$ satisfying
$ws_i > w$ (in the Bruhat order)
for all $i \in [a,b-1]$.
A bijection of $\sn$ with $\mfsmin J \times \mfs J$ is given by
the fact that
each element $w \in \sn$ factors uniquely as
\begin{equation}\label{eq:leftcosetfactor}
 w = w_- w^*,
\end{equation}
with
$(w_-,w^*) \in \mfsmin J \times \mfs J$
and $\ell(w_-) + \ell(w^*) = \ell(w)$.
This factorization satisfies the identity
\begin{equation}\label{eq:tcprod}
  T_w \wtc{s_J \ntnsp}q = T_{w_-} T_{w^*} \wtc{s_J \ntnsp}q
  = q^{\ell(w^*)} T_{w_-} \wtc{s_J \ntnsp}q.
\end{equation}
A similar factorization result employs
the set $\mfsrmin J$
of minimum length representatives of right cosets $\mfs J w$,
and defines a bijection from $\sn$ to $\mfs J \times \mfsrmin J$.

For any partition $\lambda = (\lambda_1, \dotsc, \lambda_r) \vdash n$,
the subgroup
\begin{equation*}
  \mfs{\lambda} \defeq
  \mfs{[1,\lambda_1]} \times \mfs{[\lambda_1 + 1,\lambda_1 + \lambda_2]}
  \times \cdots \times \mfs{[n-\lambda_r+1,n]} \cong
  \mfs{\lambda_1} \times \mfs{\lambda_2} \times \cdots \times \mfs{\lambda_r}
\end{equation*}
of $\sn$ generated by
\begin{equation}
\label{eq:youngsubgen}
\{ s_1, \dotsc, s_{n-1} \} \ssm 
\{ s_{\lambda_1}, s_{\lambda_1 + \lambda_2}, 
s_{\lambda_1 + \lambda_2 + \lambda_3}, 
\dotsc, 
s_{n-\lambda_r} \}
\end{equation}
is called the {\em Young subgroup} indexed by $\lambda$.
Define $H_\lambda(q)$ to be the subalgebra of $\hnq$ generated by the
corresponding elements $T_{s_i}$.
Define $\slambdamin$ to be the set of minimum length representatives
of right cosets
$\mfs \lambda u$,
i.e., the elements $u \in \sn$ satisfying
$s_iu > u$ (in the Bruhat order)
for all $s_i$ in the set (\ref{eq:youngsubgen}).
Equivalently, $\slambdamin$ consists of all $u \in \sn$ whose one-line
notations are concatenations of $r$ increasing subwords
\begin{equation}\label{eq:subwordsofu}
u_1 \cdots u_{\lambda_1}, \quad 
u_{\lambda_1 +1} \cdots u_{\lambda_1 + \lambda_2}, \quad
u_{\lambda_1 + \lambda_2 + 1} \cdots u_{\lambda_1 + \lambda_2 + \lambda_3}, \quad
\dotsc, \quad
u_{n-\lambda_r + 1} \cdots u_n.
\end{equation}
The bijections of $\sn$ with $\mfs J \times \mfsrmin J$,
for $J = [1,\lambda_1], [\lambda_1 + 1, \lambda_1 + \lambda_2], \dotsc,
[n-\lambda_r+1,n]$,
extend to a bijection
of $\sn$ with $\slambda \times \slambdamin$ and a unique factorization
of each element $w \in \sn$ as 
\begin{equation}\label{eq:rightcosetfactor}
  w = w^\circ w^-
\end{equation}
with $(w^\circ, w^-) \in \slambda \times \slambdamin$
and $\ell(w^\circ) + \ell(w^-) = \ell(w)$.
  


We call a linear function $\theta_q: H_n(q) \rightarrow \zqq$ 
an $\hnq$-{\em trace}
if it satisfies $\theta_q(gh) = \theta_q(hg)$ 
for all $g, h \in H_n(q)$.  
Such functions form a free $\zqq$-module of rank equal to the number of
partitions of $n$, and include the characters of all finite representations 
of $H_n(q)$.  The specialization at $\qp12 = 1$ of a trace is an
$\sn$-class function.
Two bases of the module of $\hnq$-trace functions are the set of
irreducible characters
$\{ \chi_q^\lambda \,|\, \lambda \vdash n \}$
and the set of induced sign characters
$\{ \epsilon_q^\lambda \,|\, \lambda \vdash n \}$,
where
\begin{equation*}
  \epsilon_q^\lambda = \sgn \ntnsp \uparrow_{H_\lambda(q)}^{\hnq},
  \qquad
  \sgn (T_{s_i}) = -1.
\end{equation*}

Define the {\em quantum matrix bialgebra} $\A = \anq$
to be to be the associative algebra with unit $1$ generated over $\zqq$
by $n^2$ variables $x=(x_{1,1},\dots,x_{n,n})$,
subject to the relations
\begin{equation}\label{eq:anqdef}
\begin{alignedat}{2}
x_{i,\ell}x_{i,k} &= \qp12x_{i,k} x_{i,\ell}, &\qquad 
x_{j,k} x_{i,\ell} &= x_{i,\ell}x_{j,k},\\
x_{j,k}x_{i,k} &= \qp12x_{i,k} x_{j,k}, &\qquad 
x_{j,\ell} x_{i,k} &= x_{i,k} x_{j,\ell} + (\qdiff) x_{i,\ell}x_{j,k},
\end{alignedat}
\end{equation}
for all indices $1 \leq i < j \leq n$ and 
$1 \leq k < \ell \leq n$.
(See \cite{ManinQNoncommG}.)
The counit map $\varepsilon(x_{i,j}) = \delta_{i,j}$, 
and coproduct map
$\Delta(x_{i,j}) = \sum_{k=1}^n x_{i,k} \otimes x_{k,j}$
give $\A$ a bialgebra structure.  
While not a Hopf algebra, 
$\A$ is closely related to 
the quantum group
$\oqslnc \cong \mathbb C \otimes \A/(\qdet(x)-1)$,
where
\begin{equation}\label{eq:qdetdef}
\qdet(x) \defeq \sum_{v \in \sn} (-\qm 12)^{\ell(v)} \permmon xv
\end{equation}
is the ($n \times n$) {\em quantum determinant} of the matrix $x = (x_{i,j})$.  
The antipode map of this Hopf algebra is 
$\mathcal S (x_{i,j}) = (-\qp12)^{j-i} \qdet(x_{[n]\ssm\{j\}, [n]\ssm \{i\}})$,
where
$x_{L,M} \defeq (x_{\ell,m})_{\ell \in L, m \in M}$,
and 
$\qdet(x_{L,M})$ is defined analogously to (\ref{eq:qdetdef}),
assuming $|L| = |M|$.
Specializing $\A$ at $\qp12 = 1$, we obtain 
the commutative 
ring $\zx$.

As a $\zqq$-module, $\A$ has a natural basis
$\{ x_{1,1}^{a_{1,1}} \cdots x_{n,n}^{a_{n,n}} \,|\, 
a_{1,1}, \dotsc, a_{n,n} \in \mathbb{N} \}$
of monomials in which variables 
appear in lexicographic order, and 
the relations (\ref{eq:anqdef})
provide an algorithm for expressing any other
monomial in terms of this basis.
The submodule $\annnq$
spanned by the monomials
$\{ x^{u,w} \defeq x_{u_1,w_1} \cdots x_{u_n,w_n} \mid u,w \in \sn \}$
has rank $n!$ and natural basis $\{ x^{e,w} \,|\, w \in \sn \}$.
Alternatively, for any fixed $u$, the set $\{ x^{u,w} \,|\, w \in \sn \}$
is a basis as well.

For any linear function $\theta_q: \hnq \rightarrow \zqq$, we may define
the generating function
\begin{equation}\label{eq:qimm}
  \imm{\theta_q}(x) \defeq \sum_{w \in \sn} \qiew \theta_q(T_w) \permmon xw
\end{equation}
in $\annnq$ which is a $q$-analog of the generating function (\ref{eq:imm}).
For
example, the generating function for the induced
sign character $\epsilon_q^\lambda$ is
\begin{equation}\label{eq:immepsilon}
  \imm{\epsilon_q^\lambda}(x) =
  \sum_I \qdet(x_{I_1,I_1}) \cdots \qdet(x_{I_r,I_r}),
\end{equation}
where
the sum is over all ordered set partitions
$I = (I_1,\dotsc,I_r)$ of $[n]$
having {\em type} $\lambda$, i.e.,
satisfying $|I_j| = \lambda_j$ \cite[Thm.\,5.4]{KSkanQGJ}.
Fixing an ordered set partition $I$ and expanding the corresponding term on
the right-hand side of (\ref{eq:immepsilon}),
we obtain
\begin{equation}\label{eq:lambepsilon}
\qdet(x_{I_1,I_1}) \cdots \qdet(x_{I_r,I_r})
  = \sum_{y \in \slambda}
(-1)^{\ell(y)} \qm{\ell(y)}2 x^{u,yu},
\end{equation}
where $u = u(I)$ is the element of $\slambdamin$
whose $r$ increasing subwords (\ref{eq:subwordsofu})
are the increasing rearrangements
of the blocks $I_1, \dotsc, I_r$ of $I$.
As $I$ varies over all ordered set partitions of type $\lambda$,
$u$ varies over all elements of $\slambdamin$ and we have
\begin{equation}\label{eq:immuslambdamin}
\imm{\epsilon_q^\lambda}(x) =
  \sum_{u \in \slambdamin} \sum_{y \in \slambda}
(-1)^{\ell(y)} \qm{\ell(y)}2 x^{u,yu}.
\end{equation}
A special case of the induced sign character immanant (\ref{eq:immepsilon})
is the generating function for the Hecke algebra sign character
$\epsilon_q^n = \chi_q^{1^n}$: the quantum determinant (\ref{eq:qdetdef}).



\section{Star networks and the evaluation map}\label{s:snets}

To each element of the form
$\smash{\wtc{s_{J_1}\ntnsp}q \cdots \wtc{s_{J_m}\ntnsp}q} \in \hnq$,
we associate a planar network $G$,
a related matrix $B$,
and a linear map $\sigma_B: \Ann \rightarrow \zqq$.
These three objects play crucial roles in the justification
of the identity (\ref{eq:evalwithsigma}) in Theorem~\ref{t:qstem},
which in turn is necessary for the justification of our main result
in Theorem~\ref{t:qepsilon}.

These three definitions are straightforward generalizations
of those given in \cite[\S 3.1--3.2]{KLSBasesQMBIndSgn}
for elements of the form
$\smash{\wtc{s_{i_1}\ntnsp}q \cdots \wtc{s_{i_m}\ntnsp}q} \in \hnq$,
i.e., where $|J_1| = \cdots = |J_m| = 2$.
Other predecessors of the
definitions, appearing in
\cite[\S 3]{CHSSkanEKL},
\cite[\S 5]{SkanNNDCB},
pertained to sequences $(J_1, \dotsc, J_m)$
of intervals satisfying the strict ``zig-zag'' conditions
stated in
\cite[\S 3]{CHSSkanEKL},
\cite[\S 3]{SkanNNDCB}.

\ssec{The network $G$ associated to $\wtc{s_{J_1}\ntnsp}q \cdots \wtc{s_{J_m}\ntnsp}q$}\label{ss:g}

To the Kazhdan-Lusztig basis element $\wtc{s_{[a,b]}\ntnsp}q$,
we associate a (simple, directed, planar) graph $G_{[a,b]}$
on $2n+1$ vertices.
For $1 \leq a < b \leq n$,
define $G_{[a,b]}$ as follows.
\begin{enumerate}
\item On the left is a column of $n$ vertices labeled
  {\em source $1$} $, \dotsc, $ {\em source $n$}, from bottom to top.
\item On the right is a column of $n$ vertices, labeled
  {\em sink $1$} $, \dotsc, $ {\em sink $n$}, from bottom to top.
\item An interior vertex is placed between the sources and sinks.
\item For $i = 1, \dotsc, a-1$ and $i = b+1, \dotsc, n$, a directed edge
  begins at source $i$ and terminates at sink $i$.
\item For $i = a, \dotsc, b$, a directed edge begins at source $i$ and
  terminates at the interior vertex, and another directed edge begins
  at the interior vertex and terminates at sink $i$.
\end{enumerate}
For $a = 1,\dotsc, n$ we define $G_{[a,a]}$
to be the similar directed planar graph on $n$ sources and
$n$ sinks with $n$ edges, each from source $i$ to sink $i$,
for $i = 1, \dotsc, n$.
Call each of the above graphs a {\em simple star network}.
In figures we will not explicitly draw vertices or show edge orientations
(assumed to be from left to right).
For $n = 4$, there are seven simple star networks:
$G_{[1,4]}$, 
$G_{[2,4]}$, 
$G_{[1,3]}$, 
$G_{[3,4]}$, 
$G_{[2,3]}$, 
$G_{[1,2]}$, 
$G_{[1,1]} = \cdots = G_{[4,4]}$, respectively,
\begin{equation}\label{eq:simplestarnets}
\begin{tikzpicture}[scale=.4,baseline=-20]
\draw[-] (0,0) -- (1,-3);
\draw[-] (0,-1) -- (1,-2);
\draw[-] (0,-2) -- (1,-1);
\draw[-] (0,-3) -- (1,0);
\end{tikzpicture}
\;, \quad
\begin{tikzpicture}[scale=.4,baseline=-20]
\draw[-] (0,0) -- (1,-2);
\draw[-] (0,-1) -- (1,-1);
\draw[-] (0,-2) -- (1,0);
\draw[-] (0,-3) -- (1,-3);
\end{tikzpicture}
\;, \quad
\begin{tikzpicture}[scale=.4,baseline=-20]
\draw[-] (0,0) -- (1,0);
\draw[-] (0,-1) -- (1,-3);
\draw[-] (0,-2) -- (1,-2);
\draw[-] (0,-3) -- (1,-1);
\end{tikzpicture}
\;, \quad
\begin{tikzpicture}[scale=.4,baseline=-20]
\draw[-] (0,0) -- (1,-1);
\draw[-] (0,-1) -- (1,0);
\draw[-] (0,-2) -- (1,-2);
\draw[-] (0,-3) -- (1,-3);
\end{tikzpicture}
\;, \quad
\begin{tikzpicture}[scale=.4,baseline=-20]
\draw[-] (0,0) -- (1,0);
\draw[-] (0,-1) -- (1,-2);
\draw[-] (0,-2) -- (1,-1);
\draw[-] (0,-3) -- (1,-3);
\end{tikzpicture}
\;, \quad
\begin{tikzpicture}[scale=.4,baseline=-20]
\draw[-] (0,0) -- (1,0);
\draw[-] (0,-1) -- (1,-1);
\draw[-] (0,-2) -- (1,-3);
\draw[-] (0,-3) -- (1,-2);
\end{tikzpicture}
\;, \quad
\begin{tikzpicture}[scale=.4,baseline=-20]
\draw[-] (0,0) -- (1,0);
\draw[-] (0,-1) -- (1,-1);
\draw[-] (0,-2) -- (1,-2);
\draw[-] (0,-3) -- (1,-3);
\end{tikzpicture}
\;. 
\end{equation}


Define a {\em star network} to be the concatenation of finitely many
simple star networks.
We write
$G \circ H$ for the network in which
sink $i$ of $G$ is identified with source $i$ of $H$, for $i = 1,\dotsc, n$.
For example, when $n = 4$ we have
\begin{equation}\label{eq:starnetsexample}
G_{[1,2]} \circ G_{[2,4]}\circ G_{[1,2]}= 
\begin{tikzpicture}[scale=.4,baseline=-20]
\draw[-] (0,0) -- (1,0) -- (2,-2) -- (3,-3);
\draw[-] (0,-1) -- (3,-1);
\draw[-] (0,-2) -- (1,-3) -- (2,-3) -- (3,-2);
\draw[-] (0,-3) -- (1,-2) -- (2,0) -- (3,0);
\end{tikzpicture}\;, \qquad
 G_{[2,4]}\circ G_{[1,3]}= 
\begin{tikzpicture}[scale=.4,baseline=-20]
\draw[-] (0,0) -- (1,-2) -- (2,-2);
\draw[-] (0,-1) -- (1,-1) -- (2,-3);
\draw[-] (0,-2) -- (1,0) -- (2,0);
\draw[-] (0,-3) -- (1,-3) -- (2,-1);
\end{tikzpicture}\;.
\end{equation}

Let $\pi = (\pi_1,\dotsc,\pi_n)$ be a sequence of source-to-sink paths
in a
wiring diagram $G$.  We call $\pi$
a {\em path family} if there exists a permutation
$w = w_1 \cdots w_n \in \sn$ such that 
$\pi_i$ is a path from source $i$ to sink $w_i$.
In this case, we say more specifically that $\pi$ has {\em type $w$}.
We say that the path family {\em covers} $G$ if it contains every edge
exactly once.

The simple star network $G_{[a,b]}$ represents the Kazhdan-Lusztig basis element
$\wtc{s_{[a,b]}\ntnsp}q$ in the sense that we have
\begin{equation*}
  \wtc{s_{[a,b]}\ntnsp}q = \sum_\pi T_{\mathrm{type}(\pi)},
\end{equation*}
where the sum is over all path families which cover $G_{[a,b]}$.
The concatenation $G_{J_1} \circ \cdots \circ G_{J_m}$ represents the product
$\wtc{s_{J_1}\ntnsp}q \cdots \wtc{s_{J_m}\ntnsp}q$ in a sense which we will
make more precise in Corollary~\ref{c:sigmastats}.

\ssec{The matrix $B$ associated to $\wtc{s_{J_1}\ntnsp}q \cdots \wtc{s_{J_m}\ntnsp}q$}\label{ss:zg}

One can enhance any planar network by
associating to each edge
a {\em weight} belonging to some ring $R$, and by defining
the {\em weight of a path} to be the product of its edge weights.
If $R$ is noncommutative,
then one multiplies weights
in the order that the corresponding edges appear in the path.
For a {\em family} $\pi = (\pi_1,\dotsc,\pi_n)$ of $n$
paths in a planar network,
one defines
$\wgt(\pi) = \wgt(\pi_1) \cdots \wgt(\pi_n)$. 
The {\em (weighted) path matrix} $B = B(G) = (b_{i,j})$ of $G$ is defined by
letting $b_{i,j}$ be the sum of weights of all paths in $G$ from source $i$
to sink $j$.
Thus the product $\permmon bw$ is equal to the sum of weights of all
path families of type $w$ in $G$ (covering $G$ or not).
It is easy to show that path matrices respect concatenation:
$B(G_1 \circ G_2) = B(G_1)B(G_2)$.
When $R$ is commutative, a result known as
{\em Lindstr\"om's Lemma}~\cite{KMG}, \cite{LinVRep}
asserts that for row and column sets $I$, $J$ with $|I| = |J|$,
the minor $\det(B_{I,J})$ is equal to the sum of weights of all nonintersecting
path families from sources indexed by $I$ to sinks indexed by $J$.

Assigning weights to the edges of $G = G_{J_1} \circ \cdots \circ G_{J_m}$
can aid in
the evaluation of a linear function $\theta_q: \hnq \rightarrow \zqq$ at
$\wtc{s_{J_1}\ntnsp}q \cdots \wtc{s_{J_m}\ntnsp}q$
by relating this evaluation to the generating function (\ref{eq:qimm}).
In particular, let
$\{ z_{h,p,k} \,|\, 1 \leq p \leq r; i_p \leq h \leq j_p; 1 \leq k \leq 2 \}$
be indeterminate weights satisfying
\begin{equation}\label{eq:quasicommz}
  z_{h_2,p_2,k_2} z_{h_1,p_1,k_1} = \begin{cases}
    z_{h_1,p_1,k_1} z_{h_2,p_2,k_2} &\text{if $p_1 \neq p_2$, or $k_1 \neq k_2$,}\\
    \qp12 z_{h_1,p_1,k_1} z_{h_2,p_2,k_2} &\text{if $p_1 = p_2$, $k_1 = k_2$, and $h_1 < h_2$.}
\end{cases}
  \end{equation}
We assign weights to the edges of $G_{J_p} = G_{[i_p,j_p]}$ as follows.
\begin{enumerate}
\item Assign weight $1$ to the $n - j_p + i_p - 1$ edges not incident
  upon the central vertex.
\item Assign weights $z_{i_p,p,1}, z_{i_p+1,p,1}, \dotsc, z_{j_p,p,1}$,
  to the $j_p - i_p + 1$ edges entering the central vertex,
  from bottom to top.
\item Assign weights $z_{i_p,p,2}, z_{i_p+1,p,2}, \dotsc, z_{j_p,p,2}$,
  to the $j_p - i_p + 1$ edges leaving the central vertex,
  from bottom to top.
\end{enumerate}
For example, the star network
$G = G_{J_1} \circ G_{J_2} = G_{[2,4]} \circ G_{[1,3]}$ is weighted as
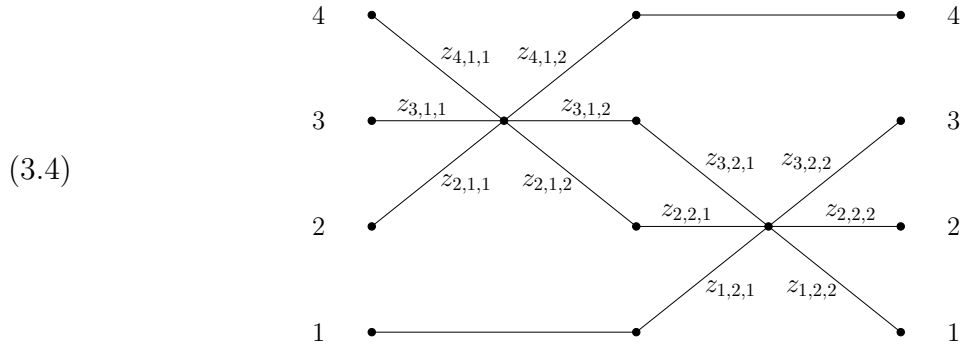
\begin{equation}\label{eq:Gbraidfig}
  \resizebox{9cm}{!}{
\raisebox{-3.2cm}{
    \begin{tikzpicture}
  \node at (-6,2) {\large $4$};
  \node at (-6,0) {\large $3$};
  \node at (-6,-2) {\large $2$};
  \node at (-6,-4) {\large $1$};
  \node at (6,2) {\large $4$};
  \node at (6,0) {\large $3$};
  \node at (6,-2) {\large $2$};
  \node at (6,-4) {\large $1$};
  \draw[-] (-5,2)
  -- (-2.5,0) node[midway, right, xshift=-1mm, yshift=2mm]{\large $z_{4,1,1}$}
  -- (0,2)  node[midway, left, xshift=1mm, yshift=2mm]{\large $z_{4,1,2}$}
  -- (5,2);
  \draw[-] (-5,0)
  -- (-2.5,0) node[midway, above, xshift=-3mm, yshift=-1mm]{\large $z_{3,1,1}$}
  -- (0,0)  node[midway, above, xshift=3mm, yshift=-1mm]{\large $z_{3,1,2}$}
  -- (2.5,-2) node[midway, right, xshift=-1mm, yshift=2mm]{\large $z_{3,2,1}$}
  -- (5,0)  node[midway, left, xshift=1mm, yshift=2mm]{\large $z_{3,2,2}$};
  \draw[-] (-5,-2)
  -- (-2.5,0) node[midway, right, xshift=-1mm, yshift=-2mm]{\large $z_{2,1,1}$}
  -- (0,-2) node[midway, left, xshift=2mm, yshift=-2mm]{\large $z_{2,1,2}$}
  -- (2.5,-2) node[midway, above, xshift=-3mm, yshift=-1mm]{\large $z_{2,2,1}$}
  -- (5,-2) node[midway, above, xshift=3mm, yshift=-1mm]{\large $z_{2,2,2}$};
  \draw[-] (-5,-4)
  -- (0,-4) 
  -- (2.5,-2) node[midway, right, xshift=-1mm, yshift=-2mm] {\large $z_{1,2,1}$}
  -- (5,-4) node[midway, left, xshift= 2mm, yshift=-2mm] {\large $z_{1,2,2}$};
\filldraw (-5,0) circle (.7mm);
\filldraw (0,0) circle (.7mm);
\filldraw (5,0) circle (.7mm);
\filldraw (-5,2) circle (.7mm);
\filldraw (0,2) circle (.7mm);
\filldraw (5,2) circle (.7mm);
\filldraw (-5,-2) circle (.7mm);
\filldraw (0,-2) circle (.7mm);
\filldraw (5,-2) circle (.7mm);
\filldraw (-5,-4) circle (.7mm);
\filldraw (0,-4) circle (.7mm);
\filldraw (5,-4) circle (.7mm);
\filldraw (-2.5,0) circle (.7mm);
\filldraw (2.5,-2) circle (.7mm);            
\end{tikzpicture}
}}
\end{equation}
and the weighted path matrix of $G$ is
\begin{equation}\label{eq:Bex}
  \begin{aligned}
    B &=
  \begin{bmatrix}
    1 & 0              & 0              & 0               \\
    0 & z_{2,1,1}z_{2,1,2} & z_{2,1,1}z_{3,1,2} & z_{2,1,1}z_{4,1,2} \\
    0 & z_{3,1,1}z_{2,1,2} & z_{3,1,1}z_{3,1,2} & z_{3,1,1}z_{4,1,2} \\
    0 & z_{2,1,1}z_{2,1,2} & z_{2,1,1}z_{3,1,2} & z_{2,1,1}z_{4,1,2} 
  \end{bmatrix}\ntksp\ntksp
  \begin{bmatrix}
    z_{1,2,1}z_{1,2,2} & z_{1,2,1}z_{2,2,2} & z_{1,2,1}z_{3,2,2} & 0 \\
    z_{2,2,1}z_{1,2,2} & z_{2,2,1}z_{2,2,2} & z_{2,2,1}z_{3,2,2} & 0 \\
    z_{3,2,1}z_{1,2,2} & z_{3,2,1}z_{2,2,2} & z_{3,2,1}z_{3,2,2} & 0 \\
    0              & 0              & 0               & 1 
  \end{bmatrix}\\
  &= \begin{bmatrix}
    z_{1,1,1}z_{1,2,2} & z_{1,2,1}z_{2,2,2} & z_{1,2,1}z_{3,2,2} & 0 \\
    z_{2,1,1}(z_D + z_U)z_{1,2,2}
    & z_{2,1,1}(z_D + z_U)z_{2,2,2}
    & z_{2,1,1}(z_D + z_U)z_{3,2,2}
    & z_{2,1,1}z_{4,1,2} \\
    z_{3,1,1}(z_D + z_U)z_{1,2,2}
    & z_{3,1,1}(z_D + z_U)z_{2,2,2}
    & z_{3,1,1}(z_D + z_U)z_{3,2,2}
    & z_{3,1,1}z_{4,1,2} \\
    z_{4,1,1}(z_D + z_U)z_{1,2,2}
    & z_{4,1,1}(z_D + z_U)z_{2,2,2}
    & z_{4,1,1}(z_D + z_U)z_{3,2,2}
    & z_{4,1,1}z_{4,1,2} 
  \end{bmatrix},
  \end{aligned}
\end{equation}
where we have defined $z_D = z_{2,1,2}z_{2,2,1}$, $z_U = z_{3,1,2}z_{3,2,1}$.

\ssec{The map $\sigma_B$ associated to
$\wtc{s_{J_1}\ntnsp}q \cdots \wtc{s_{J_m}\ntnsp}q$}\label{ss:sigmab}

The final, algebraic, object which we associate to
the product $\wtc{s_{J_1}\ntnsp}q \cdots \wtc{s_{J_m}\ntnsp}q$
is a map $\sigma_B$ which allows us to
evaluate a linear functional $\theta_q$ at this product
by considering the corresponding generating function
$\imm{\theta_q}(x) \in \annnq$.
This strategy can be quite helpful when one has a simpler expression for
$\imm{\theta_q}(x)$ than for its coefficients
$\theta_q(T_w)$.
In our case, an expression for $\imm{\epsilon_q^\lambda}(x)$ is given by
(\ref{eq:immepsilon}), while we have no explicit formula for
$\epsilon_q^\lambda(T_w)$.

Let $Z_G$ be the quotient of the noncommutative ring
\begin{equation*}
    \mathbb Z[\qp12, \qm12]
    \langle z_{h_p,p,k}
    \,|\, p=1,\dotsc, r; h_p = i_p, \dotsc, j_p; k = 1,2 \rangle
    \end{equation*}
modulo the ideal generated by the
relations (\ref{eq:quasicommz}), and assume that $\qp12$, $\qm12$
commute with all other indeterminates.
Let $z_G$
be the product of all indeterminates $z_{h_p,p,k}$, in lexicographic order,
and for $f \in Z_G$,
let $[z_G]f$ denote the coefficient of $z_G$ in $f$.
Let $B$ be the weighted path matrix of $G$.
We define $\sigma_B$ to be the
$\zqq$-linear map
\begin{equation}\label{eq:sigmadef}
\begin{aligned}
  \sigma_B: \annnq &\rightarrow \zqq\\
\permmon xv &\mapsto [z_G] \permmon bv,
\end{aligned}
\end{equation}
where $[z_G] \permmon bv$ denotes the coefficient of $z_G$ in $\permmon bv$,
taken after $\permmon bv$ 
is expanded in the lexicographic basis of $Z_G$.
Note that the ``substitution'' $x_{i,j} \mapsto b_{i,j}$ is performed
only for monomials of the form $x^{e,v}$ in $\Ann$:
we define $\sigma_B(x^{u,w})$ by first expanding $x^{u,w}$
in the basis $\{ x^{e,v} \,|\, v \in \sn \}$, and {\em then} performing
the substitution.

To illustrate the utility of the function $\sigma_B$,
we consider both sides of the identity (\ref{eq:evalwithsigma})
when
the product of Kazhdan-Lusztig basis elements is
\begin{equation}\label{eq:sigmaexample}
\wtc{s_{[2,4]}\ntnsp}q \wtc{s_{[1,3]}\ntnsp}q = (1+q)\sum_{w \leq 3421} T_w \in H_4(q),
\end{equation}
and when the linear function 
$\theta_q: H_4(q) \rightarrow \zqq$
is defined by
\begin{equation*}
  \theta_q(T_w) = \begin{cases}
    1 &\text{if $w = 3412$},\\
    -1 &\text{if $w = 4312$},\\
    0 &\text{otherwise}.
  \end{cases}
  \end{equation*}
To compute the left-hand side of (\ref{eq:evalwithsigma}), we
evaluate
\begin{equation}\label{eq:LHS}
\theta_q(\wtc{s_{[2,4]}\ntnsp}q\wtc{s_{[2,4]}\ntnsp}q ) = (1+q)(1) + (0)(-1) = 1+q,
\end{equation}
since in (\ref{eq:sigmaexample})
$T_{3412}$ has coefficient $1+q$ and $T_{4312}$ has coefficient $0$.
To compute the right-hand side of (\ref{eq:evalwithsigma}), we first write
\begin{equation}\label{eq:eximm}
  \imm{\theta_q}(x) = q^{-2} x_{1,3}x_{2,4}x_{3,1}x_{4,2}
  - \qm52 x_{1,4}x_{2,3}x_{3,1}x_{4,2},
\end{equation}
since $\ell(3412) = 4$ and $\ell(4312) = 5$.
Now observe that the star network
and weighted path matrix corresponding to
$\wtc{s_{[2,4]}\ntnsp}q\wtc{s_{[2,4]}\ntnsp}q$ are
$G = G_{[2,4]} \circ G_{[1,3]}$ in (\ref{eq:Gbraidfig})
and $B$ in (\ref{eq:Bex}). Thus we compute
$\sigma_B(\imm{\theta_q}(x))$ by substituting
$x_{i,j} = b_{i,j}$
in (\ref{eq:eximm})
to obtain
\begin{equation}\label{eq:zGex}
  \begin{aligned}
q^{-2}b_{1,3}b_{2,4}b_{3,1}b_{4,2} &= q^{-2} z_{1,2,1}z_{3,2,2}z_{2,1,1}z_{4,1,2}
z_{3,1,1}(z_D + z_U)z_{1,2,2}z_{4,1,1}(z_D + z_U)z_{2,2,2},\\
\qm52 b_{1,4}b_{2,3}b_{3,1}b_{4,2} &= 0,
\end{aligned}
\end{equation}
and extract the coefficient of
$z_G = z_{1,2,1} z_{1,2,2} z_{2,1,1} \cdots z_{4,1,2}$ from these.
Since $z_D^2$ and $z_U^2$ are not square-free, we ignore two monomials
in the expansion of (\ref{eq:zGex}).  Now, using the relations
(\ref{eq:quasicommz}) to express the remaining two monomials
in lexicographic order, we introduce multiples of $\qp12$
for pairs
$(z_{h_1,p,k}, z_{h_2,p,k})$ of variables sharing second and third indices,
\begin{equation*}
  \begin{aligned}
  z_{3,2,2}z_{4,1,2}z_D &= \qp12 z_Dz_{3,2,2}z_{4,1,2},\\
  z_{3,2,2}z_{4,1,2}z_U &= \qp12 z_Uz_{3,2,2}z_{4,1,2},
  \end{aligned}
  \qquad
  \begin{aligned}
  z_{3,2,2}z_{1,2,2} &= \qp12 z_{1,2,2}z_{3,2,2},\\
  z_{3,2,2}z_{2,2,2} &= \qp12 z_{2,2,2}z_{3,2,2},
  \end{aligned}
  \qquad
  z_U z_D = q z_D z_U.
\end{equation*}
Specifically, we introduce a multiple of $q^2$ for the monomial in which
$z_D$ precedes $z_U$, and a multiple of $q^3$ in the other monomial.
Thus we have
\begin{equation*}
  [z_G] q^{-2} (q^2 + q^3)
  z_{1,2,1}z_{1,2,2}z_{2,1,1}z_D z_{2,2,2}z_{3,1,1}z_U z_{3,2,2}z_{4,1,1}z_{4,1,2} = 1+q,
\end{equation*}
which matches (\ref{eq:LHS}),
as desired.

The maps $\sigma$ behave well with respect to concatenation
and matrix products.
In particular, the result \cite[Prop.\,3.4]{KLSBasesQMBIndSgn}
applies verbatim to the networks, matrices, and maps
generalized as above from their earlier definitions in
\cite{KLSBasesQMBIndSgn}.

\begin{prop}\label{p:sigmacoprod}
  Let star networks
$G$, $H$
have weighted path matrices $B$, $C$, respectively. 
Then for all $u,w \in \sn$ we have
\begin{equation}\label{eq:sigmacoprod}
\sigma_{BC}(x^{u,w}) = \sum_{v \in \sn} \sigma_B(x^{u,v}) \sigma_C(x^{v,w}).
\end{equation}
\end{prop}
\begin{proof}
  Extend the proof of
  \cite[Prop.\,3.4]{KLSBasesQMBIndSgn}
  to include elements of the ring $Z_G$, as defined in Subsection~\ref{ss:zg}.
\end{proof}

  \section{$G$-tableaux and the combinatorics of the evaluation map $\sigma_B$}
  \label{s:gtabx}

  Justification of the identity (\ref{eq:evalwithsigma}),
  which allows us to prove Theorem~\ref{t:qepsilon}, depends upon
  a combinatorial interpretation of evaluations $\sigma_B(x^{u,v})$,
  which is stated in Proposition~\ref{p:sigmastats}.  This result
  and the preparatory Lemmas~\ref{l:sigmareduced} -- \ref{l:decomposestat}
  extend earlier work
  in \cite[\S 5--6]{CHSSkanEKL} and
  \cite[\S 3--4]{KLSBasesQMBIndSgn}
  on special cases of the evaluations.

First we observe two cases in which such an evaluation
  is quite simple.
\begin{lem}\label{l:sigmareduced}
  Let star network $G = G_{J_1} \circ \cdots \circ G_{J_m}$ have weighted path
  matrix $B$, and fix $u, v \in \sn$.
  \begin{enumerate}
  \item If no path family of type $u^{-1}v$ covers $G$,
    then $\sigma_B(x^{u,v}) = 0$.
  \item If exactly one path family of type $v$ covers $G$,
    then $\sigma_B(x^{e,v}) = \qev$.
  \end{enumerate}
  \end{lem}
  \begin{proof}
    (1) If no path family of type $u^{-1}v$ covers $G$,
    then we have $b_{u_1,v_1} \cdots b_{u_n,v_n} = 0$. 

    \noindent
    (2) Suppose that there is exactly one path family $\pi$ of type $v$
    which covers $G$.  Then for $i < j$, the pair $(v_i, v_j)$
    is an inversion in $v$ if and only if the paths $\pi_i$ and $\pi_j$ cross.
    Since $\pi$ is the unique path family of type $v$, the paths cannot cross
    more than once.  If $\pi_i$, $\pi_j$ do not intersect, or intersect without
    crossing, then each variable in $\wgt(\pi_i)$ is 
    lexicographically less than each variable in $\wgt(\pi_j)$. 
    Relative to one another, all of these variables appear
    in lexicographic order
    in the expression
    \begin{equation}\label{eq:wgtpi}
      \wgt(\pi) = \wgt(\pi_1) \cdots \wgt(\pi_n)
    \end{equation}
    and contribute $q^0$ to the coefficient of $z_G$.
    On the other hand, if $\pi_i$, $\pi_j$ do intersect at the central
    vertex of $G_{J_p}$,
    then for some indices $a < b$, 
    the variable $z_{b,p,2}$ appears in $\wgt(\pi_i)$,
    earlier in (\ref{eq:wgtpi}) than the lexicographically lesser variable
    $z_{b,p,2}$ which appears in $\wgt(\pi_j)$.  By (\ref{eq:quasicommz})
    this pair of variables contributes $\qp12$ to the coefficient of $z_G$.
  \end{proof}

  In other cases, the evaluation $\sigma_B(x^{u,v})$ is less simple,
  and we will interpret such an evaluation by covering a star network
\begin{equation}\label{eq:starnetfactor}
  G = G_{J_1} \circ \cdots \circ G_{J_m}
  \end{equation}
  with a path family $\pi = (\pi_1, \dotsc, \pi_n)$
  and arranging the component paths into a (French) Young diagram
  whose shape is some partition $\lambda$ of $n$.  Following
  \cite{CHSSkanEKL}, we call the resulting structure a {\em $G$-tableau},
  or more specifically a {\em $\pi$-tableau} of {\em shape $\lambda$}.
  If $\type(\pi) = w$, we say that the tableau has type $w$.
Since a $\pi$ tableau contains each path of $\pi$ exactly once,
and since there are no constraints on where paths can be placed,
there are always $n!$ $\pi$-tableaux of a fixed shape $\lambda$.
  
Given a $\pi$-tableau $U$ of shape $\lambda$,
  we define two Young tableaux $L(U)$ and $R(U)$ of shape $\lambda$
  by replacing each path in $\pi$ by its source and sink index, respectively.
A fixed path family $\pi$ of type $v$ and a permutation
$u \in \sn$ determine a path tableau $U(\pi,u,uv) = \pi_{u_1} \cdots \pi_{u_n}$
which satisfies
$L(U(\pi, u, uv)) = u_1 \cdots u_n$ and
$R(U(\pi, u, uv)) = (uv)_1 \cdots (uv)_n$.
The inclusion of $uv$ in our notation $U(\pi,u,uv)$ is superfluous but makes
clear the ordering of sinks as they appear in the tableau.

We define statistics on path families and on the tableaux containing them. 
Suppose that two paths $\pi_i$, $\pi_j$ pass through the central vertex of
some star network $G_{J_p}$ in the factorization (\ref{eq:starnetfactor}).
We call the triple $(\pi_i, \pi_j, p) = (\pi_j, \pi_i, p)$
a {\em crossing} of $\pi$ if the two paths cross there,
and a {\em noncrossing} otherwise.
Define
\begin{equation}\label{eq:cross}
  \cross(\pi) =
\# \{ (\pi_i, \pi_j, p) \,|\,
  (\pi_i, \pi_j, p) \text{ is a crossing} \}.
\end{equation}
Suppose that the triple $(\pi_i, \pi_j, p)$ is a noncrossing,
with $\pi_i$ entering and leaving the central vertex of $G_{J_p}$ below
$\pi_j$.  If $U$ is any $\pi$-tableau, then we call the triple
an {\em inverted noncrossing of $U$} if $\pi_j$
appears in an earlier column of $U$
than does $\pi_i$.
Define
\begin{equation}\label{eq:incross}
  \incross(U) =
\# \{ (\pi_i, \pi_j, p) \,|\,
  (\pi_i, \pi_j, p) \text{ is an inverted noncrossing in } U \}.
\end{equation}

Now using the statistics $\cross$ and $\incross$, 
we can combinatorially evaluate
$\sigma_C(x^{u,v})$ in the case that
$G$ is a simple star network with weighted path matrix $C$.

\begin{lem}\label{l:sigmareversalinterp}
 Let simple star network $G_J$ have weighted path matrix $C$.
  For $u,w \in \sn$ we have
  \begin{equation}\label{eq:sigmareversalinterp}
    \sigma_C(x^{u,w}) =
  \begin{cases}
    \qp{\cross(\pi)}2 q^{\incross(U(\pi,u,w))} &\text{if $u^{-1}w \in \mfs J$},\\
    0 &\text{otherwise},
    \end{cases}
  \end{equation}
  where $\pi$ is the unique path family
  of type $u^{-1}w$ covering $G_J$.
\end{lem}
\begin{proof}
  If $u^{-1}w \notin \mfs J$, then there is no path family
  of type $u^{-1}w$ which covers $G_J$ and the right-hand side of
  (\ref{eq:sigmareversalinterp}) is $0$.
  By Lemma~\ref{l:sigmareduced}, the left-hand side is $0$ as well.
  Assume therefore that we have $u^{-1}w \in \mfs J$.
  
  Suppose $u = e$. Then the left-hand side of (\ref{eq:sigmareversalinterp})
  is $\qp{\ell(w)}2$ by
  Lemma~\ref{l:sigmareduced}.
  The right-hand side is also
  $\qp{\cross(\pi)}2 q^{\incross(U(\pi,u,w))} = \qp{\ell(w)}2$,
  since the left tableau
  $u_1 \cdots u_n = 1 \cdots n$
  guarantees that $U(\pi,u,w)$ has no inverted noncrossings.

  Now suppose the claim to be true when $\ell(u) = 0,\dotsc, k$ and consider
  $u$ of length $k+1$.  Factor $u$ as $u = s_iu'$ with $\ell(u') = k$.
  By \cite[Eq.\,(2.5)]{KLSBasesQMBIndSgn}, we have
  \begin{equation*}
    \sigma_C(x^{u,w}) =
    \sigma_C(x^{s_iu', w}) = 
    \begin{cases}
      \sigma_C(x^{u',s_iw}) &\text{if $s_iw > w$},\\
      \sigma_C(x^{u',s_iw}) + (\qdiff)\sigma_C(x^{u',w}) &\text{if $s_iw < w$}.
    \end{cases}
  \end{equation*}
  By induction, and since $(u')^{-1}s_iw = u^{-1}w$, this is
  \begin{equation}\label{eq:sigmaccombin}
    \begin{cases}
      \qp{\cross(\pi)}2 q^{\incross(U(\pi,u',s_iw))} &\text{if $s_iw > w$},\\
      \qp{\cross(\pi)}2 q^{\incross(U(\pi,u',s_iw))}
      + (\qdiff)(\qp{\cross(\pi')}2 q^{\incross(U(\pi', u',w))})
      &\text{if $s_iw < w$},
     \end{cases}
  \end{equation}
  where $\pi$, as in the statement of the lemma,
  is the unique path family of type $(u')^{-1}s_iw = u^{-1}w$ covering $G_J$,
  and where $\pi'$ is the unique path family of type $(u')^{-1}w = u^{-1}s_iw$
  covering $G_J$.

  Let us consider the two $G_J$-tableaux appearing in (\ref{eq:sigmaccombin}),
  and relate their crossings and inverted noncrossings to those of $U(\pi,u,w)$.

  The first tableau,
  \begin{equation*}
    U(\pi, u', s_iw) = \pi_{u_1} \cdots \pi_{u_{i-1}} \pi_{u_{i+1}} \pi_{u_i} \pi_{u_{i+2}} \cdots \pi_{u_n},
  \end{equation*}
  differs from $U(\pi, u, w)$ only by the transposition of paths
  $\pi_{u_i}$ and $\pi_{u_{i+1}}$.
 Thus if
 $s_iw > w$, then 
  we have $u_i > u_{i+1}$ and $w_i < w_{i+1}$,
  i.e., the paths $\pi_{u_i}$ and $\pi_{u_{i+1}}$ cross.
  It follows that we have $\incross(U(\pi,u',s_iw)) = \incross(U(\pi,u,w))$, and
  the claim is true.
  On the other hand, if $s_iw < w$, then
  we have $u_i > u_{i+1}$ and $w_i > w_{i+1}$,
  i.e., the paths
  $\pi_{u_i}$ and $\pi_{u_{i+1}}$ form
  a (noninverted) noncrossing in $U(\pi,u',s_iw)$
  and an inverted noncrossing in $U(\pi,u,w)$.
  It follows that we have
  \begin{equation}\label{eq:firstincrossid}
    \incross(U(\pi,u',s_iw)) = \incross(U(\pi,u,w)) - 1.
    \end{equation}

  The second tableau,
  \begin{equation*}
    \begin{aligned}
      U(\pi', u', w) &= \pi'_{u_1} \cdots \pi'_{u_{i-1}} \pi'_{u_{i+1}} \pi'_{u_i} \pi'_{u_{i+2}} \cdots \pi'_{u_n} \\
      &= \pi_{u_i} \cdots \pi_{u_{i-1}} \pi'_{u_{i+1}} \pi'_{u_i} \pi_{u_{i+2}} \cdots \pi_{u_n},
    \end{aligned}
  \end{equation*}
  differs from the tableau $U(\pi,u,w)$ only in the pair
  $(\pi'_{u_{i+1}}, \pi'_{u_i})$ from sources $(u_{i+1}, u_i)$ to
  sinks $(w_i, w_{i+1})$, respectively.
  This pair forms a crossing in $U(\pi',u',w)$ in place of the 
  inverted noncrossing of $(\pi_{u_i}, \pi_{u_{i+1}})$ in $U(\pi,u,w)$.
  It is straightforward
  but tedious
  to check that
  all other contributions to the differences
  $\cross(\pi) - \cross(\pi')$ and
  $\incross(U(\pi,u,w)) - \incross(U(\pi',u', w))$
  are due to paths $\pi_{u_j}$ from sources $u_j$ to sinks $w_j$ with indices
  $j$ belonging to
  \begin{equation*}
  Q_i \defeq \{ j \,|\, u_{i+1} < u_j < u_i; \ w_{i+1} < w_j < w_i \}.
  \end{equation*}
  Each such path $\pi_{u_j}$ crosses
  neither $\pi_{u_i}$ nor $\pi_{u_{i+1}}$ in $\pi$,
  but crosses both $\pi'_{u_i}$ and $\pi'_{u_{i+1}}$ in $\pi'$.
  Similarly, for each such path $\pi_{u_j}$, exactly one of the pairs
  $(\pi_{u_j}, \pi_{u_i})$, $(\pi_{u_j}, \pi_{u_{i+1}})$
  forms an inverted noncrossing in $\pi$
  while neither of the pairs
  $(\pi_{u_j}, \pi'_{u_i})$, $(\pi_{u_j}, \pi'_{u_{i+1}})$
  forms an inverted noncrossing in $\pi'$.
  Thus we have
  \begin{equation}\label{eq:secondincrossid}
    \begin{gathered}
      \cross(\pi') = \cross(\pi) + 1 + 2|Q_i|,\\
      \incross(U(\pi',u',w)) = \incross(U(\pi,u,w)) - 1 - |Q_i|.
    \end{gathered}
    \end{equation}

  In summary, when $s_iw < w$, we combine (\ref{eq:firstincrossid}) and
  (\ref{eq:secondincrossid}) to see that
  the expression (\ref{eq:sigmaccombin})
  is equal to
  \begin{equation*}
    \begin{aligned}
      &\qp{\cross(\pi)}2 q^{\incross(U(\pi,u,w)) - 1} + (\qdiff)\qp{\cross(\pi')+1 + 2|Q_i|}2 q^{\incross(U(\pi,u,w)) - 1 - |Q_i|} \\
      &\quad = \qp{\cross(\pi)}2 q^{\incross(U(\pi,u,w))}
      ( q^{-1} + (\qdiff) \qp12 \qp{2|Q_i|}2 q^{-1} q^{-|Q_i|})\\
      &\quad = \qp{\cross(\pi)}2 q^{\incross(U(\pi,u,w))}( q^{-1} + 1 - q^{-1}),  
    \end{aligned}
  \end{equation*}
  and again the claim is true.
\end{proof}

Alternatively, when $G = G_J$ is a simple star network with weighted
path matrix $C$, we can algebraically evaluate $\sigma_C(x^{u,w})$
in terms of the
factorizations in (\ref{eq:leftcosetfactor}).
\begin{lem}\label{l:sigmareversalinv}
  Let simple star network $G_J$ have weighted path matrix $C$,
  and let $u,w \in \sn$ have the unique $\mfsmin J \times \mfs J$
  factorizations $u = u_- u^*$, $w = w_-w^*$. 
  Then
  we have
  \begin{equation*}
    \sigma_C(x^{u,w}) = \begin{cases}
      \qustar \qwstar &\text{if $u^{-1}w \in \mfs J$},\\
      0 &\text{otherwise.}
    \end{cases}
  \end{equation*}
\end{lem}
\begin{proof}
By Lemma~\ref{l:sigmareversalinterp} it is sufficient to show that
$\qustar \qwstar  = \qp{\cross(\pi)}2 q^{\incross(U(\pi,u,w))}$, 
equivalently
\begin{equation}\label{eq:simplestarlem}
  \inv(u^*) + \inv(w^*) = \cross(\pi) + 2 \, \incross(U(\pi,u,w)),
\end{equation}
whenever $u^{-1}w \in \mfs J$.

Consider a pair $(i,j)$ with $i < j$ and the contributions of this pair
and the triple $(\pi_i,\pi_j,1)$ to both sides of (\ref{eq:simplestarlem}).
If the paths do not intersect, then $i$ or $j$ does not belong to $J$
and we have a contribution of $0$ to both sides of (\ref{eq:simplestarlem}).
Assume therefore that the paths intersect.

If the triple $(\pi_i,\pi_j,1)$ is a crossing, then $(i,j)$ is an inversion in
exactly one of $u^*$ and $w^*$.
Thus we have a contribution of $1$ to both sides of (\ref{eq:simplestarlem}).
On the other hand, if $(\pi_i,\pi_j,1)$ is a noncrossing,
then $\pi_j$ intersects $\pi_i$ from above.
Thus $(\pi_i,\pi_j,1)$ is an inverted noncrossing if and only if $\pi_j$
precedes $\pi_i$ in $U(\pi,u,w)$, i.e., if and only if $j$ precedes $i$ in $u$.
But this condition is equivalent to $j$ preceding $i$ in $w$
since $\pi_i$ and $\pi_j$ do not cross.
Hence we have a contribution of $2$ to both sides of (\ref{eq:simplestarlem})
or a contribution of $0$ to both sides.
\end{proof}


To combinatorially evaluate expressions of the form $\sigma_B(x^{u,w})$
where $B$ is the weighted path matrix of an arbitrary star network $G$,
we will use the fact that $G$ is a concatenation of smaller star networks.
When a star network $G$ can be decomposed as a concatenation $G' \circ H$,
it is natural to similarly decompose any path family
$\pi = (\pi_1, \dotsc, \pi_n)$ that covers $G$
as $\pi = \pi^{G'} \circ \pi^H$, where $\pi^{G'} = (\pi_1^{G'}, \dotsc, \pi_n^{G'})$,
$\pi^H = (\pi_1^H, \dotsc, \pi_n^H)$,
and $\type(\pi) = \type(\pi^{G'}) \type(\pi^H)$.
One can also decompose a single path $\pi_i$ as $(\pi_i)^{G'} \circ (\pi_i)^H$
but one must be careful however: if $\mathrm{type}(\pi^{G'}) = y$ then
we have $(\pi_i)^{G'} = (\pi^{G'})_i = \pi_i^{G'}$,
but $(\pi_i)^H = (\pi^H)_{y_i}$, which is not in general equal to
$(\pi^H)_i = \pi_i^H$,
the component path of $\pi^H$ which starts at source $i$ of $H$.



The above decompositions of path families
and path tableaux
satisfy some simple identities.

\begin{lem}\label{l:decomposestat}
  Let path family $\pi$ cover star network $G = G' \circ H$,
  and let $\pi^{G'}$, $\pi^H$ be its restrictions to $G'$, $H$.  Then
  we have
  \begin{equation*}
    \cross(\pi) = \cross(\pi^{G'}) + \cross(\pi^H).
  \end{equation*}
  Now fix $u \in \sn$ and define $w = u \cdot \mathrm{type}(\pi)$,
  $v = u \cdot \mathrm{type}(\pi^{G'})$. Then we have
  \begin{equation*}
    \incross(U(\pi,u,w)) =
    \incross(U(\pi^{G'}\ntnsp,u,v)) + \incross(U(\pi^H\ntnsp,v,w)).
  \end{equation*}
\end{lem}
\begin{proof}
  Write $G' = G_{J_1} \circ \cdots \circ G_{J_k}$,
  $H = G_{J_{k+1}} \circ \cdots \circ G_{J_m}$ and let
  paths $\pi_i$, $\pi_j$ 
  intersect at the central vertex of some factor $G_{J_p}$.
  Suppose that this point of intersection is a crossing.
  Then it is a crossing of $\pi^{G'}$ if $p \leq k$ and a crossing
  of $\pi^H$ if $p > k$.
  Thus it contributes $1$ to $\cross(\pi)$
  and $1$ to either $\cross(\pi^{G'})$ (if $p \leq k$)
  or $\cross(\pi^H)$ (if $p > k$).

  Suppose on the other hand that the point of intersection is a noncrossing.
  Then it is a noncrossing of $\pi^{G'}$ if $p \leq k$ and
  a noncrossing of $\pi^H$ if $p > k$.  Moreover, $\pi_i$ enters
  and departs from the central vertex of $G_{J_p}$ above $\pi_j$
  if and only if the restrictions of these paths to $G'$ or $H$ do the same.
  Now consider the positions of these paths in their respective tableaux.
  Let $a$, $b$, be the positions of $\pi_i$ and $\pi_j$ in $U(\pi,u,w)$:
  $\pi_i = \pi_{u_a}$,
  $\pi_j=\pi_{u_b}$, and assume without loss of generality that $a < b$.
  Then the restrictions of the two paths to $G'$ and $H$ also appear
  in positions $a$ and $b$, for we have
%

\vspace{2mm}
\hhhpsp
\begin{center}
\newcolumntype{R}{>{$}c<{$}}
\begin{tabularx}{157mm}{|R|R|R|}%
\hline
& \phantom , &\\
\mbox{tableau}
& \mbox{path in position $a$}
& \mbox{path in position $b$} \\
& \phantom , &\\
\hline
& \phantom t &\\
U(\pi,u,w)
&\pi_{u_a} \mbox{ from source $u_a$ of $G$ to sink $w_a$}
&\pi_{u_b} \mbox{ from source $u_b$ of $G$ to sink $w_b$}\\
& \phantom t &\\
U(\pi^{G'},u,v)
&\pi_{u_a}^{G'} \mbox{ from source $u_a$ of $G'$ to sink $v_a$}
&\pi_{u_b}^{G'} \mbox{ from source $u_b$ of $G'$ to sink $v_b$}\\
& \phantom t &\\
U(\pi^H,v,w)
&\pi_{v_a}^H \mbox{ from source $v_a$ of $H$ to sink $w_a$}
&\pi_{u_b}^H \mbox{ from source $v_b$ of $H$ to sink $w_b$}\\
& \phantom t &\\
\hline
\end{tabularx}
\end{center}
\vspace{2mm}
\ssp
with
$\pi_{u_a} = \pi_{u_a}^{G'} \circ \pi_{v_a}^H$,
$\pi_{u_b} = \pi_{u_b}^{G'} \circ \pi_{v_b}^H$.
Thus the triple
$(\pi_{u_a}, \pi_{u_b}, p)$
contributes $1$ to $\incross(U(\pi,u,w))$
if and only if
$(\pi_{u_a}^{G'}, \pi_{u_b}^{G'}, p)$
contributes $1$ to $\incross(U(\pi^{G'},u,v))$ ($p \leq k$),
or
$(\pi_{u_a}^H, \pi_{u_b}^H, p)$
contributes $1$ to $\incross(U(\pi^H,v,w))$ ($p > k$).
\end{proof}

Combining the simple star network evaluation $\sigma_C(x^{u,w})$
from Lemma~\ref{l:sigmareversalinterp}
with the identities in Lemma~\ref{l:decomposestat},
we can now inductively prove a combinatorial evaluation of $\sigma_B(x^{u,w})$
where $B$ is the weighted path matrix of an arbitrary star network.

\begin{prop}\label{p:sigmastats}
  Let star network $G$ have weighted path matrix $B$.
For $u,w \in \sn$ we have 
\begin{equation}\label{eq:stats}
\sigma_B (x^{u,w}) = 
\sum_\pi \qp{\cross(\pi)}2
q^{\incross(U)},
\end{equation}
where the sum is over path families $\pi$ of type $u^{-1}w$ covering $G$,
and $U = U(\pi,u,w)$ is the unique $\pi$-tableaux of shape $(n)$
satisfying $L(U) = u_1 \cdots u_n$, $R(U) = w_1 \cdots w_n$.
\end{prop}
\begin{proof}
  By Lemma~\ref{l:sigmareversalinterp}, the claim is true when $G$ is
  the simple star network $G_{[i_1,j_1]}$.
    
  Now suppose the claim is true for $G$ a concatenation of
  $1, \dotsc, m$ simple star networks
  and consider $G = G' \circ H$ where $G' = G_{J_1} \circ \cdots \circ G_{J_m}$
  and $H = G_{J_{m+1}}$.
    Let $B'$, $C$ be the path matrices of $G'$, $H$, respectively
    so that $B = B'C$ is the path matrix of $G$.

By Proposition~\ref{p:sigmacoprod}
    and induction we have
\begin{equation}\label{eq:sigmalong}
  \sigma_{B'C}(x^{u,w})
  =
\sum_{t \in \mfs{J_{m+1}}}
  \sum_{\pi^{G'}}
  \qp{\cross(\pi^{G'})}2 q^{\incross(U(\pi^{G'}\ntksp,u,wt^{-1}))}
  \qp{\cross(\pi^H)}2 q^{\incross(U(\pi^H,wt^{-1},w))},
  \end{equation}
  where $\pi^{G'}$ varies over all path families of type $u^{-1}wt^{-1}$
  which cover $G'$.
  Since pairs of path families of types $u^{-1}wt^{-1}$ and $t$ covering
  $G'$ and $H$, respectively, correspond to path families of type $u^{-1}w$
  covering $G$, we may apply
  Lemma~\ref{l:decomposestat} to the right-hand side above to obtain
  \begin{equation*}
  \sum_\pi
  \qp{\cross(\pi)}2 q^{\incross(U(\pi,u,w))},
  \end{equation*}
  where the sum is over all path families $\pi$ of type $u^{-1}w$ which cover
  $G$.
\end{proof}

\section{Products $\wtc{s_{J_1}\ntnsp}q \cdots \wtc{s_{J_m}\ntnsp}q$ 
and the $q$-immanant evaluation theorem}\label{s:qimmevalthm}





We complete our justification of
the crucial identity (\ref{eq:evalwithsigma}) in Theorem~\ref{t:qstem},
first using Proposition~\ref{p:qewcoeff} to relate
evaluations of the form $\sigma_B(x^{e,w})$ to
coefficients in the natural expansion of 
$\wtc{s_{J_1}\ntnsp}q \cdots \wtc{s_{J_m}\ntnsp}q$.
As a corollary, we obtain a second interpretation of the above coefficients
which generalizes a result of Deodhar.
The first two results generalize
\cite[Prop.\,3.6 -- Thm.\,3.7]{KLSBasesQMBIndSgn}.


\begin{prop}\label{p:qewcoeff}
  Let star network $G = G_{J_1} \circ \cdots \circ G_{J_m}$
  have weighted path matrix $B$, and fix $w \in \sn$.
Then 
the coefficient of $T_w$ in
$\wtc{s_{J_1}\ntnsp}q \cdots \wtc{s_{J_m}\ntnsp}q$ is
equal to $\qiew \sigma_B(x^{e,w})$.
\end{prop}
\begin{proof}
  Consider the simple star network $G = G_{J_1}$
  and its weighted path matrix $B$.
  By Lemma~\ref{l:sigmareversalinv}
  we have
\begin{equation*}
\sigma_B(x^{e,w}) =
\begin{cases}
\qew &\text{if $w \in \mfs{J_1}$},\\
0    &\text{otherwise}.
\end{cases}
\end{equation*}
Since $\wtc{s_{J_1}\ntnsp}q = \sum_{w \in \mfs{J_1}} T_w$,
the result is true for any simple star network.
	
Now assume that the result holds for concatenations 
of $1,\dotsc, m-1$ simple star networks,
and define $\{ a_w \,|\, w \in \sn \} \subseteq \mathbb Z[q]$ by
\begin{equation*}
\wtc{s_{J_1}\ntnsp}q \cdots \wtc{s_{J_{m-1}}\ntnsp}q = \sum_{w \in \sn} a_w T_w.
\end{equation*}
Consider the star network $G = G_{J_1} \circ \cdots \circ G_{J_m}$
with weighted path matrix $B$ and decompose $G$ as $G' \circ H$,
where $G' = G_{J_1} \circ \cdots \circ G_{J_{m-1}}$
has weighted path matrix $B'$ and
$H = G_{J_m}$ has weighted path matrix $C$.
By Proposition~\ref{p:sigmacoprod} we have
\begin{equation*}
  \sigma_B(x^{e,w}) =
  \sum_{u \in \mfs{J_m}} \sigma_{B'}(x^{e,wu^{-1}}) \sigma_C(x^{wu^{-1},w}),
\end{equation*}
where the sum is taken over $u \in \mfs{J_m}$ by Lemma~\ref{l:sigmareversalinv}.
Now factor $w$ as $w_- w^*$ with $w_- \in \mfsmin{J_m}$, $w^* \in \mfs{J_m}$,
and define $t = w^* u^{-1}$ so that we have $wu^{-1} = w_-w^* (w^*)^{-1}t = w_-t$.
Making this substitution in the sum
and observing that as $u$ varies over $\mfs{J_m}$, so does $t$,
we apply Lemma~\ref{l:sigmareversalinv}
to the second factor in each summand to obtain
\begin{equation*}
\sum_{t \in \mfs{J_m}} \ntksp \sigma_{B'} (x^{e,w_-t})  \sigma_C (x^{w_-t,w})
= \sum_{t \in \mfs{J_m}} \ntksp \sigma_{B'} (x^{e,w_-t})  \qwmtstar \qwstar
= \, \qwstar \ntksp \sum_{t \in \mfs{J_m}} \ntksp \sigma_{B'}(x^{e,w_-t}) \qet.
\end{equation*}
By induction, this is
\begin{equation*}
\qwstar \ntksp \sum_{t \in \mfs{J_m}} \ntksp \qwmt a_{w_-t} \qet
= \, \qwstar \ntksp \sum_{t \in \mfs{J_m}} \ntksp \qwm \qet a_{w_-t} \qet
= \, \qew \ntksp \sum_{t \in \mfs{J_m}}\ntksp a_{w_-t} q^{\ell(t)}.
\end{equation*}	
On the other hand, consider the element
\begin{equation}\label{eq:element}
\wtc{s_{J_m}\ntnsp}q \cdots \wtc{s_{J_m}\ntnsp}q
= \Big( \sum_{w \in \sn} a_w T_w \Big) \wtc{s_{J_m}\ntnsp}q
= \sum_{w \in \sn} a_w T_{w_-} T_{w^*} \wtc{s_{J_m}\ntnsp}q.
\end{equation}
By (\ref{eq:klbasis}), (\ref{eq:tcprod})
and the fact that reversals \avoidp, this is
\begin{equation*}
  \sum_{y \in \mfsmin{J_m}} \ntksp T_y
  \ntksp \sum_{t \in \mfs{J_m}} \ntksp a_{yt} T_t \wtc{s_{J_m}\ntnsp}q
= \sum_{y \in \mfsmin{J_m}} \ntksp T_y \wtc{s_{J_m}\ntnsp}q
\ntksp \sum_{t \in \mfs{J_m}} \ntksp a_{yt} q^{\ell(t)} 
= \sum_{y \in \mfsmin{J_m}} \ntksp T_y
\ntksp \sum_{v \in \mfs{J_m}} \ntksp T_v
\ntksp \sum_{t \in \mfs{J_m}} \ntksp a_{yt} q^{\ell(t)}.
\end{equation*}
Now
$\qew$ times the coefficient of $T_w$ in this expression is
\begin{equation*}
\qew \sum_{t \in \mfs{J_m}} a_{w_-t} q^{\ell(t)}. \qedhere
\end{equation*}
\end{proof}

Now we have a quick proof of the
$q$-immanant evaluation theorem
for star networks.
\begin{thm}\label{t:qstem}
  Let $\theta_q: \hnq \rightarrow \zqq$ be linear, and let
  star network $G = G_{J_1} \circ \cdots \circ G_{J_m}$
  have weighted path matrix $B$.
  Then we have
\begin{equation}\label{eq:qstem}
  \theta_q(\wtc{s_{J_1}\ntnsp}q \cdots \wtc{s_{J_m}\ntnsp}q)
  = \sigma_B(\imm{\theta_q}(x)).
\end{equation}
\end{thm}
\begin{proof}
  Write 
$\wtc{s_{J_1}\ntnsp}q \cdots \wtc{s_{J_m}\ntnsp}q = \sum_{w \in \sn} a_w T_w$.
  Then the right-hand side of (\ref{eq:qstem}) is
\begin{equation*}
 \sigma_B \Big(\sum_{w \in \sn} \theta_q(T_w) \qiew x^{e,w} \Big)
  = \sum_{w \in \sn} \theta_q(T_w) \qiew \sigma_B(x^{e,w})
  = \sum_{w\in \sn} \theta_q(T_w) a_w
  = \theta_q \Big(\sum_{w \in \sn} a_wT_w \Big),
\end{equation*}
where the second equality follows from Proposition~\ref{p:qewcoeff}.
But this is the left-hand side of (\ref{eq:qstem}).
\end{proof}

Proposition~\ref{p:qewcoeff} and the special case $u = e$ 
of Proposition~\ref{p:sigmastats} yield a proof of
a generalization of Deodhar's defect formula \cite[Prop.\,3.5]{Deodhar90}
for coefficients in the expansion of $(1 + T_{s_{i_1}}) \cdots (1 + T_{s_{i_m}})$, 
which we will use in Section~\ref{s:indsgneval}.
Let $\pi = (\pi_1, \dotsc, \pi_n)$ be a path family covering a star network
$G = G_{J_1} \circ \cdots \circ G_{J_m}$.
If two paths $\pi_i$, $\pi_j$ intersect at the central vertex
of $G_{J_p}$, call the triple $(\pi_i, \pi_j, p)$
{\em defective} or {\em a defect}
if the paths
have previously crossed an odd number of times (i.e., in
$G_{J_1}, \dotsc, G_{J_{i_{p-1}}}$).
Let $\dfct(\pi)$ denote the number of defects of $\pi$,
\begin{equation}\label{eq:dfctdef}
  \dfct(\pi) = \# \{(\pi_i, \pi_j, p) \,|\,
  (\pi_i, \pi_j, p) \text{ defective } \}.
\end{equation}
The original definition of $\dfct(\pi)$ given in \cite[Defn.\,2.2]{Deodhar90}
(as interpreted in \cite[Rmk.\,6]{BWHex})
and used throughout \cite{KLSBasesQMBIndSgn}
is simply the number of {\em indices} $p$ appearing in defective triples
$(p, \pi_i, \pi_j)$, since each simple star $G_{J_p} = G_{[j_p,j_p+1]}$
considered in those papers completely determines the
two paths to intersect at its central vertex.
The proof of the following result is nearly identical to that of
\cite[Cor.\,4.2]{KLSBasesQMBIndSgn}.
We include it here, both for the convenience of the reader,
and because it relies upon
Propositions~\ref{p:sigmastats} and \ref{p:qewcoeff} in this paper,
which differ from the corresponding (weaker) results
\cite[Prop.\,4.1]{KLSBasesQMBIndSgn},
\cite[Prop.\,3.6]{KLSBasesQMBIndSgn}
in that paper.

\begin{cor}\label{c:sigmastats}
  The coefficients in the expansion
$\wtc{s_{J_1}\ntnsp}q \cdots \wtc{s_{J_m}\ntnsp}q = \underset w \sum a_w T_w$
are given by
\begin{equation*}
  a_w = \sum_\pi q^{\dfct(\pi)},
\end{equation*}
where the sum is over all path families of type $w$ which cover the star network
$G_{J_1} \circ \cdots \circ G_{J_m}$.
\end{cor}
\begin{proof} 
  Substitute $u = e$ in Proposition~\ref{p:sigmastats}.
  The right-hand side of (\ref{eq:stats}) is a sum over
  path families $\pi$ of type $w$, and
  each tableau $U = U(\pi,e,w)$ is simply
  the sequence $\pi = (\pi_1,\dotsc,\pi_n)$.
  Thus an inverted noncrossing of $U$ is simply a noncrossing
  in which the upper path has index less than that of the lower path, i.e.,
  a defective noncrossing.  To relate these to all defects,
  let us temporarily define
  \begin{equation*}
    \begin{aligned}
      \dnc(\pi) &= \text{ number of defective noncrossings of $\pi$},\\
      \dc(\pi) &= \text{ number of defective crossings of $\pi$},
    \end{aligned}
  \end{equation*}
  so that $\dfct(\pi) = \dnc(\pi) + \dc(\pi)$.
  Now observe that for any path familiy of type $w$ we have
    \begin{equation*}
    \cross(\pi) = \inv(w) + 2\dc(\pi)
    \end{equation*}
    because if paths $\pi_a$, $\pi_b$ cross $k$ times,
    then at most one of those crossings contributes to $\inv(w)$,
    while exactly half of the remaining crossings are defective.
    Thus the right-hand side of (\ref{eq:stats}) becomes
    \begin{equation}\label{eq:rhsstats}
      \sum_\pi \qp{\inv(w) + 2\dc(\pi)}2 q^{\dnc(\pi)}
      = \sum_\pi \qew q^{\dc(\pi)+\dnc(\pi)} = \qew \sum_\pi q^{\dfct(\pi)},
    \end{equation}
    where the sum is over all path families of type $w$ which cover the
    star network $G_{J_1} \circ \cdots \circ G_{J_m}$.
    
    By Proposition \ref{p:qewcoeff}, the left-hand side of (\ref{eq:stats})
    is $\qew a_w$.
  Combining this with (\ref{eq:rhsstats}), we have the desired result.
\end{proof}

\section{Evaluation of induced sign characters}\label{s:indsgneval}

We now complete the solution of Problem~\ref{prob:interp} by providing
a combinatorial interpretation
in Theorem~\ref{t:qepsilon}
for all expressions of the form
\begin{equation}\label{eq:mainprob}
  \epsilon_q^\lambda(\wtc{s_{J_1}\ntnsp}q \cdots \wtc{s_{J_m}\ntnsp}q),
\end{equation}
with $\lambda \vdash n$ and $(J_1,\dotsc,J_m)$ a sequence
of subintervals of $[n]$.
Generalizing \cite[Thm.\,5.4]{KLSBasesQMBIndSgn},
Theorem~\ref{t:qepsilon} relies upon several lemmas and definitions
which are straightforward generalizations of those given in
\cite[\S 5]{KLSBasesQMBIndSgn}.
On the other hand, the fact that
not all reversals appearing as subscripts in (\ref{eq:mainprob}) are
adjacent transpositions complicates matters significantly,
requiring a new involution (\ref{eq:zeta}) and four new identities
collected in Lemma~\ref{l:statids}.

By Theorem~\ref{t:qstem} and Equation (\ref{eq:immuslambdamin}), 
each expression (\ref{eq:mainprob}) is equal to
\begin{equation}\label{eq:epsilonimmcombined}
  \sigma_B(\imm{\epsilon_q^\lambda}(x))
  = \sum_{u \in \smash{\slambdamin}} \sum_{y \in \slambda} (-1)^{\ell(y)} \qm{\ell(y)}2
\sigma_B(x^{u,yu}),
\end{equation}
where $B$ is the weighted path matrix of the star network
$G = G_{J_1} \circ \cdots \circ G_{J_m}$.
Thus it suffices to combinatorially interpret the right-hand side of
(\ref{eq:epsilonimmcombined}).
To do this, we will apply Proposition~\ref{p:sigmastats}
to each expression $\sigma_B(x^{u,yu})$
and will compute statistics for tableaux belonging to the set
\begin{equation*}
    \mathcal U_I = \mathcal U_I(G)
    \defeq \{ U(\pi, u, yu ) \,|\, \pi \text{ covers } G, u = u(I), 
    y \in \slambda \},
\end{equation*}
where $I \mapsto u(I)$ is the bijection defined after
(\ref{eq:lambepsilon}).
Note that our restriction on $y$ forces the sink indices of paths in
components
\begin{equation}\label{eq:component}
  (\lambda_1 + \cdots + \lambda_{k-1}+1), \dotsc, (\lambda_1 + \cdots + \lambda_k)
\end{equation}
of $U(\pi,u,yu)$ to be a permutation of the source indices of the same paths.

Following the outline of \cite[\S 5]{KLSBasesQMBIndSgn},
we will apply a sign-reversing involution
to the tableaux in each set $\mathcal U_I$ by
defining an involution on a second set
$\mathcal T_I$ of tableaux, in obvious bijection with $\mathcal U_I$.
For a wiring diagram $G$
let $\mathcal T_I = \mathcal T_I(G)$ be 
the set of all column-closed, left column-strict $G$-tableaux 
$W$ of shape $\lambda^\tr$
such that $L(W^\tr)_k = I_k$ (as sets) for $k = 1,\dotsc,r$.
Define the bijection
\begin{equation}\label{eq:delta}
  \begin{aligned}
    \delta = \delta_I\ntksp: \mathcal U_I &\rightarrow \mathcal T_I\\
    U &\mapsto W,
  \end{aligned}
\end{equation}
as in \cite[\S 5]{KLSBasesQMBIndSgn},
by letting $W$ be the
left column-strict $G$-tableau of shape $\lambda^\tr$
whose $k$th column
consists of entries
(\ref{eq:component})
of $U$.

Since $U \in \mathcal U_I$ and $\delta(U) \in \mathcal T_I$
contain the same path family, it is easy to see
that $\delta$ does not affect the statistic $\cross$.
On the other hand, it changes the statistic $\incross$ in a very simple way,
as in \cite[Lem.\,5.1]{KLSBasesQMBIndSgn}.
Define $\cdncross(U)$ to be the number of defective noncrossings of pairs of 
paths appearing in the same column of $U$.

\begin{lem}\label{l:TU}
  Let $I$ be an ordered set partition.
  For $U \in \mathcal U_I$ we have
\begin{equation}\label{eq:incdn}
\incross(U) = \incross(\delta(U)) + \cdncross(\delta(U)).
\end{equation}
\end{lem}
\begin{proof}
  Follow the proof of \cite[Lem.\,5.1]{KLSBasesQMBIndSgn},
  merely replacing each noncrossing {\em index} as defined in
  \cite{KLSBasesQMBIndSgn})
  with a noncrossing {\em triple} as defined
  preceding (\ref{eq:cross}).
\end{proof}

Now we define the involution
\begin{equation}\label{eq:zeta}
 \zeta = \zeta_I\ntksp: \mathcal T_I \rightarrow \mathcal T_I
\end{equation}
(different from $\zeta$ in \cite{KLSBasesQMBIndSgn})
as follows.  Let
$G = G_{J_1} \circ \cdots \circ G_{J_m}$,
$u = u(I)$, and
$\lambda = \type(I)$. Fix $y \in \slambda$
and let
$W = \delta(U(\pi,u,yu))$.
\begin{enumerate}
\item If $W$ is column-strict,
  then define $\zeta(W) = W$.
\item Otherwise, 
  \begin{enumerate}
  \item Let $t$ be the greatest index such that
    column $t$ of $W$ is not column-strict.
  \item Let $p$ be the greatest index such that (at least) two paths
    in column $t$ of $W$ pass through
    the interior vertex of $G_{J_p}$.
  \item Let $j$, $j'$ be the two positions in the interval
    $[\lambda_1 + \cdots + \lambda_{t-1} + 1, \lambda_1 + \cdots \lambda_t]$
    for which paths $\pi_{u_j}$, $\pi_{u_{j'}}$ have the 
    right-to-left lexicographically greatest pair $((yu)_j,(yu)_{j'})$
    of sink indices. 
  \item Let $\hat \pi = (\hat \pi_1, \dotsc, \hat \pi_n)$ be the path family
    obtained from $\pi$ by swapping the terminal subpaths of
    $\pi_{u_j}$ and $\pi_{u_{j'}}$, beginning at the interior vertex of $G_{J_p}$.
    ($\hat \pi_i = \pi_i$ for $i \notin \{u_j,u_{j'}\}.$)
\item
  Define $\zeta(W)$ to be the tableau obtained from $W$ by replacing
$\pi_i$ by $\hat \pi_i$, for $i = 1,\dotsc,n$.
  \end{enumerate}
\end{enumerate}

If $W \in \mathcal T_I$ is not a fixed point of $\zeta$, then the tableaux
$U, \widehat U \in \mathcal U_I$ satisfying
$\delta(U) = W$, $\delta(\widehat U) = \zeta(W)$
are closely related.

\begin{lem}\label{l:lengthdiff}
  Let $I = (I_1, \dotsc I_r)$ be an ordered set partition of type $\lambda$,
  and define $u = u(I)$.
  Let $G$-tableaux $W \in \mathcal T_I$ and $U, \widehat U \in \mathcal U_I$
  satisfy $W = \delta(U) \neq \zeta(W) = \delta(\widehat U)$,
  and define path families $\pi$, $\hat \pi$ as above.
Then for some generator $s \in \slambda$ and
some permutations
$y, \hat y = ys \in \slambda$
we have
$U = U(\pi, u, yu)$, $\widehat U = U(\hat \pi, u, \hat y u)$.
\end{lem}
\begin{proof}
  The tableaux $U$, $\widehat U$ contain the same path
  families as $W$ and $\zeta(W)$,
  respectively,
  and these path families are $\pi$, $\hat \pi$, as defined in
  the definition of $\zeta$.

  Listing the source indices of paths in column $t$ of $W$, from bottom to top,
  we obtain the increasing sequence
  \begin{equation}\label{eq:incsourceseq}
    (u_{\lambda_1 + \cdots + \lambda_{t-1}+1}, u_{\lambda_1 + \cdots + \lambda_{t-1}+2}, \dotsc,
    u_{\lambda_1 + \cdots + \lambda_{t}}).
    \end{equation}
  Since elements of $\mathcal T_I$ are column-closed, the sink indices of the
  paths form a permutation of this sequence,
  \begin{equation*}\label{eq:sinkseq}
((yu)_{\lambda_1 + \cdots + \lambda_{t-1}+1}, (yu)_{\lambda_1 + \cdots + \lambda_{t-1}+2}, \dotsc,
    (yu)_{\lambda_1 + \cdots + \lambda_{t}}),
  \end{equation*}
  for some $y \in \slambda$.
  In particular, the sink indices $(yu)_j$ and $(yu)_{j'}$ of
  $\pi_{u_j}$, $\pi_{u_{j'}}$ are simply
  two components $u_i$ and $u_{i'}$ of (\ref{eq:incsourceseq}).

  By our choice of $p$, no two paths in column $t$ of $W$ intersect to
  the right of $G_{J_p}$.  By our choice of $j$ and $j'$, no path in this column
  has a sink index strictly between $u_i$ and $u_{i'}$. It follows that
  $|i - i'| = 1$.  Thus the sink indices of the path family $\hat \pi$ are
  given by $\hat yu = ysu$, where $s$ is the adjacent transposition
  $s_{\min\{i,i'\}} \in \slambda$.
%
\end{proof}


Furthermore, when $W \in \mathcal T_I$ is not a fixed point of $\zeta$,
the values of the statistics $\cross$, $\incross$ and $\cdncross$ on
$W$ and $\zeta(W)$ are closely related, as are $\ell(y)$ and $\ell(\hat y)$,
where $y$, $\hat y$ are defined as in Lemma~\ref{l:lengthdiff}.

\begin{lem}\label{l:statids}
  Let $W = U(\pi,u,yu) \neq \zeta(W) = U(\hat \pi, u, \hat yu)$ be as above
  with $\hat \pi$ being defined by
  the intersection of $\pi_{u_j}$ and $\pi_{u_{j'}}$
  in $G_{J_p}$.  Let $b$ be the number of paths in $\pi$
  (equivalently, in $\hat \pi$) which enter $G_{J_p}$ between the paths
  $\pi_{u_j}$ and $\pi_{u_{j'}}$, and which leave $G_{J_p}$ between
  the same two paths.
  Then we have
    \begin{gather}
      \ell(\hat y) = \begin{cases}
        \ell(y) - 1 &\text{if $(\pi_{u_j}, \pi_{u_{j'}}, p)$ is a proper crossing or defective noncrossing}, \\
        \ell(y) + 1 &\text{if $(\pi_{u_j}, \pi_{u_{j'}}, p)$ is a proper noncrossing or defective crossing};
        \end{cases}\label{eq:yid}\\
      \cross(\hat \pi) = \begin{cases}
        \cross(\pi) - 1 - 2b &\text{if $(\pi_{u_j}, \pi_{u_{j'}}, p)$ is a crossing},\\
        \cross(\pi) + 1 + 2b &\text{otherwise};
      \end{cases}\label{eq:crossid}\\
      \incross(\zeta(W)) = \begin{cases}
        \incross(W) + b &\text{if $(\pi_{u_j}, \pi_{u_{j'}}, p)$ is a crossing},\\
        \incross(W) - b &\text{otherwise};
      \end{cases}\label{eq:incrossid}\\
      \cdncross(\hat \pi) = \begin{cases}
        \cdncross(\pi) &\text{if $(\pi_{u_j}, \pi_{u_{j'}}, p)$ is proper},\\
        \cdncross(\pi) - 1 &\text{if $(\pi_{u_j}, \pi_{u_{j'}}, p)$ is a defective noncrossing},\\
        \cdncross(\pi) + 1 &\text{if $(\pi_{u_j}, \pi_{u_{j'}}, p)$ is a defective crossing}.\label{eq:cdncrossid}
      \end{cases}
    \end{gather}
\end{lem}
\begin{proof}
(Identity (\ref{eq:yid})) 
  Recall that we have
  \begin{equation*}
    \begin{gathered}
      \text{source}(\pi_{u_j}) = \text{source}(\hat \pi_{u_j}) = u_j,\\
      \text{source}(\pi_{u_{j'}}) = \text{source}(\hat \pi_{u_{j'}}) = u_{j'},
    \end{gathered} \qquad 
    \begin{aligned}    
      &\text{sink}(\pi_{u_j}) = \text{sink}(\hat \pi_{u_{j'}}) = (yu)_j,\\
      &\text{sink}(\pi_{u_{j'}}) = \text{sink}(\hat \pi_{u_j}) = (yu)_{j'} > (yu)_j.
    \end{aligned}
  \end{equation*}
By Lemma~\ref{l:lengthdiff}, we have
$y, \hat y = ys \in \slambda$ and therefore $\ell(\hat y) = \ell(y) \pm 1$.
Since $u \in \slambdamin$, we also have
\begin{equation}\label{eq:yu}
  \ell(yu) = \ell(y) + \ell(u), \qquad \ell(\hat yu) = \ell(\hat y) + \ell(u).
  \end{equation}
If $(\pi_{u_j}, \pi_{u_{j'}}, p)$ is a proper crossing or a defective noncrossing,
then the relative orders of $(u_j, u_{j'})$ and $( (yu)_j, (yu)_{j'} )$
are different, i.e., $u_{j'} < u_j$.  By the definition of $j, j'$ and
(\ref{eq:incsourceseq}) we also have $j' < j$.
Thus the pair of letters $( (yu)_j, (yu)_{j'} )$ is inverted in the word
$yu$ and not in the word $\hat yu$.  It follows that $\ell(yu) > \ell(\hat yu)$
and by (\ref{eq:yu}) that $\ell(y) > \ell(\hat y) = \ell(y) - 1$.
The other case of (\ref{eq:yid}) is similar.  

To justify the identities (\ref{eq:crossid}) -- (\ref{eq:cdncrossid}),
we first look closely at paths
$\pi_i = \hat \pi_i$, $i \not \in \{j, j'\}$, which pass
through the interior vertex of $G_{J_p}$.  We partition such paths into nine
equivalence classes according to whether they enter $G_{J_p}$ below
the both of the paths $\pi_{u_j}$, $\pi_{u_{j'}}$ (equivalently, below both
$\hat \pi_{u_j}$, $\hat \pi_{u_{j'}}$) 
between the two paths, or above both paths, and whether they leave $G_{J_p}$
below, between, or above the two paths.  We name the nine classes
$AD, AE, AF, BD, BE, BF, CD, CE, CF$, according to the (magnified) diagrams
of paths intersecting in $G_{J_p}$
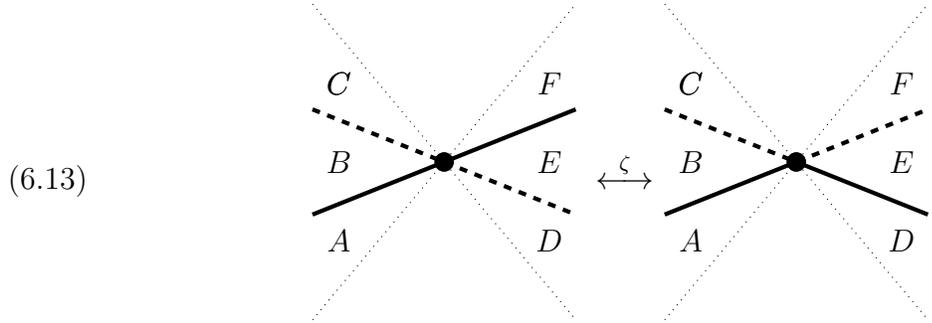
	\begin{equation}
	\begin{tikzpicture}[scale=0.7,baseline=-70]
	\draw[dotted] (0,0) -- (5,-6);
	\draw[dashed,ultra thick] (0,-2) -- (5,-4) ;
	\draw[-,ultra thick] (0,-4) -- (5,-2) ;
	\draw[dotted] (0,-6) -- (5,0);
	\draw[-,fill] (2.5,-3) circle (5pt);
	\node at (0.5,-1.5) {$C$};
	\node at (0.5,-1.5) {$C$};
	\node at (4.5,-1.5) {$F$};
	\node at (0.5,-3) {$B$};
	\node at (4.5,-3) {$E$};
	\node at (0.5,-4.5) {$A$};
	\node at (4.5,-4.5) {$D$};
	\end{tikzpicture}\; \overset{\large\zeta}{\longleftrightarrow}
	\begin{tikzpicture}[scale=0.7,baseline=-70]
\draw[dotted] (0,0) -- (5,-6);
\draw[dashed,ultra thick] (0,-2) -- (2.5,-3) -- (5,-2) ;
\draw[-,ultra thick] (0,-4) -- (2.5,-3) -- (5,-4) ;
\draw[dotted] (0,-6) -- (5,0);
\draw[-,fill] (2.5,-3) circle (5pt);
\node at (0.5,-1.5) {$C$};
\node at (0.5,-1.5) {$C$};
\node at (4.5,-1.5) {$F$};
\node at (0.5,-3) {$B$};
	\node at (4.5,-3) {$E$};
	\node at (0.5,-4.5) {$A$};
	\node at (4.5,-4.5) {$D$};
	\end{tikzpicture}\;.
	\end{equation}
The solid and dashed lines represent edges from
$\pi_j$, $\pi_{j'}$, $\hat\pi_j$ and $\hat\pi_{j'}$,
and the dotted lines the edges with highest and lowest sink
or source indices among edges incident with the central vertex of $G_{J_p}$.
Every other edge meeting this vertex lies in one of the six labeled regions. 
We may have $\pi_{u_j}$, $\pi_{u_{j'}}$ on the left, with
$\hat \pi_{u_j}$, $\hat \pi_{u_{j'}}$ on the right, or vice versa.
We may have the paths indexed by $u_j$ entering $G_{J_p}$
above the others in both diagrams, or below the others in both diagrams.
All four combinations are possible.
Thus, paths in class $AE$ enter $G_{J_p}$ below the two bold paths and leave
$G_{J_p}$ between them, while the class $BE$ has cardinality $b$
as defined in the statement of the lemma.

\smallskip
(Identity (\ref{eq:crossid}))
Consider the contributions of all points of intersection
$(\pi_i, \pi_{i'}, k)$
and $(\hat \pi_i, \hat \pi_{i'}, k)$ to $\cross(\pi)$ and $\cross(\hat \pi)$,
respectively.
For $k < p$ and all $i, i'$, we have
that
$(\pi_i, \pi_{i'}, k)$ is a crossing if and only if
$(\hat \pi_i, \hat \pi_{i'}, k)$ is a crossing.
The same is true for $k \geq p$,
provided that $i, i' \not \in \{ u_j, u_{j'} \}$.
For $k > p$ and $i \notin \{ u_j, u_{j'} \}$,
crossings $(\pi_i, \pi_{u_j}, k)$ and $(\hat \pi_i, \hat \pi_{u_{j'}}, k)$ 
correspond bijectively, as do
crossings $(\pi_i, \pi_{u_{j'}}, k)$ and $(\hat \pi_i, \hat \pi_{u_j}, k)$. 
By the definition of $j, j'$, there are no crossings of the form
$(\pi_{u_j}, \pi_{u_{j'}}, k)$ or
$(\hat \pi_{u_j}, \hat \pi_{u_{j'}}, k)$ for $k > p$.
So far, contributions to $\cross(\pi)$ and $\cross(\hat \pi)$ are equal.

Now consider points of intersection of the form
\begin{equation}\label{eq:triples}
  (\pi_i, \pi_{u_j}, p), \qquad (\pi_i, \pi_{u_{j'}}, p),
\end{equation}
for $i \not \in \{ u_j, u_{j'} \}$,
and their images under $\zeta$,
\begin{equation}\label{eq:hattriples}
  (\hat \pi_i, \hat \pi_{u_j}, p), \qquad (\hat \pi_i, \hat \pi_{u_{j'}}, p).
\end{equation}
If $\pi_i \in AF \cup CD$, then all four triples are crossings;
if $\pi_i \in AD \cup CF$, then none is.
If $\pi_i \in AE \cup BD \cup BF \cup CE$, then
exactly one of the triples (\ref{eq:triples}) is a crossing, as is
exactly one of the triples (\ref{eq:hattriples}).
On the other hand, if $\pi_i$ is one of the $b$ paths in $BE$, then
either
both triples (\ref{eq:triples}) are crossings
while both triples (\ref{eq:hattriples}) are noncrossings
(if $(\pi_{u_j}, \pi_{u_{j'}}, p)$ is a crossing),
or 
both triples (\ref{eq:triples}) are noncrossings
while both triples (\ref{eq:hattriples}) are crossings
(if $(\pi_{u_j}, \pi_{u_{j'}}, p)$ is a noncrossing).
Finally
by the definition of $\zeta$, exactly one of the two triples
\begin{equation}\label{eq:piuj}
(\pi_{u_j}, \pi_{u_{j'}}, p), \qquad (\hat \pi_{u_j}, \hat \pi_{u_{j'}}, p)
\end{equation}
is a crossing and the other is a noncrossing.
Thus the points of intersection of the forms
(\ref{eq:triples}), (\ref{eq:hattriples}), (\ref{eq:piuj})
contribute a surplus of $2b + 1$ to $\cross(\pi)$
if $(\pi_{u_j}, \pi_{u_{j'}}, p)$ is a crossing,
and to $\cross(\hat \pi)$ otherwise.

\smallskip
(Identity (\ref{eq:incrossid}))
For $k < p$ and all $i, i'$, we have
that
$(\pi_i, \pi_{i'}, k)$ is an inverted noncrossing if and only if
$(\hat \pi_i, \hat \pi_{i'}, k)$ is an inverted noncrossing.
The same is true for $k \geq p$,
provided that $i, i' \not \in \{ u_j, u_{j'} \}$.
For $k > p$ and $i \notin \{ u_j, u_{j'} \}$,
noncrossings $(\pi_i, \pi_{u_j}, k)$ and $(\hat \pi_i, \hat \pi_{u_{j'}}, k)$ 
correspond bijectively, as do
noncrossings $(\pi_i, \pi_{u_{j'}}, k)$ and $(\hat \pi_i, \hat \pi_{u_j}, k)$. 
Moreover, since paths indexed $u_j$, $u_{j'}$ all appear in the same column $t$,
this correspondence preserves inversion of noncrossings. 
By the definition of $j, j'$, there are no noncrossings of the form
$(\pi_{u_j}, \pi_{u_{j'}}, k)$ or
$(\hat \pi_{u_j}, \hat \pi_{u_{j'}}, k)$ for $k > p$.
So far, contributions to $\incross(\pi)$ and $\incross(\hat \pi)$ are equal.

Now consider points of intersection of the forms
(\ref{eq:triples}), (\ref{eq:hattriples}) for $i \notin \{ u_j, u_{j'} \}$.
If $\pi_i \in AF \cup CD$, then all four triples are crossings;
if $\pi_i \in AD \cup CF$, then all four are noncrossing.
Moreover, if we have $\pi_i \in AD$ appearing to the left of column $t$
or $\pi_i \in CF$ appearing to the right of column $t$, then
all four noncrossings are inverted.  Otherwise none of the four is inverted.
If $\pi_i \in AE \cup BD \cup BF \cup CE$, then
exactly one of the triples (\ref{eq:triples}) is a noncrossing, as is
exactly one of the triples (\ref{eq:hattriples}).
These two noncrossings are inverted if we have
$\pi_i \in AE \cup BD$ appearing to the right of column $t$,
or if we have
$\pi_i \in BF \cup CE$ appearing to the left of column $t$.
Otherwise, the two noncrossings are not inverted.
If $\pi_i \in BE$, then either
both triples (\ref{eq:triples}) are crossings
while both triples (\ref{eq:hattriples}) are not
(if $(\pi_{u_j}, \pi_{u_{j'}}, p)$ is a crossing),
or vice versa
(if $(\pi_{u_j}, \pi_{u_{j'}}, p)$ is a noncrossing).
In both cases, exactly one of the two noncrossings is inverted.
Finally, observe that neither triple in
(\ref{eq:piuj}) can be an inverted noncrossing, since all four paths
appear in column $t$.
Thus the points of intersection of the forms
(\ref{eq:triples}), (\ref{eq:hattriples}), (\ref{eq:piuj})
contribute a surplus of $b$ to $\incross(\pi)$
if $(\pi_{u_j}, \pi_{u_{j'}}, p)$ is a crossing,
and to $\cross(\hat \pi)$ otherwise.

\smallskip
(Identity (\ref{eq:cdncrossid}))
For $k < p$ and all $i, i'$, we have
that
$(\pi_i, \pi_{i'}, k)$ is a column defective noncrossing
if and only if
$(\hat \pi_i, \hat \pi_{i'}, k)$ is a column defective noncrossing.
The same is true for $k \geq p$,
provided that $i, i' \not \in \{ u_j, u_{j'} \}$.
For $k > p$ and $i \notin \{ u_j, u_{j'} \}$,
no triple $(\pi_i, \pi_{u_j}, k)$, $(\pi_i, \pi_{u_{j'}}, k)$,
$(\hat \pi_i, \hat \pi_{u_j}, k)$, $(\hat \pi_i, \hat \pi_{u_{j'}}, k)$ 
can be a column defective noncrossing, since the definition of $j, j'$
guarantees that no path $\pi_i$ can belong to column $t$ and intersect 
paths indexed by $u_j$, $u_{j'}$ to the right of $G_{J_p}$.
So far, contributions to $\cdncross(\pi)$ and $\cdncross(\hat \pi)$ are equal.

Now consider points of intersection of the forms
(\ref{eq:triples}), (\ref{eq:hattriples}) for $i \notin \{ u_j, u_{j'} \}$.
If $\pi_i \in AF \cup BF \cup CF \cup AE \cup BE \cup CE$,
then it cannot appear in column $t$ by our choice of $j, j'$.
If $\pi_i \in CD$, then all four triples are crossings.
If $\pi_i \in AD \cup BD$, then
$(\pi_i, \pi_{u_j}, p)$
is a column defective noncrossing if and only if
$(\hat \pi_i, \hat \pi_{u_j}, p)$ is,
and
$(\pi_i, \pi_{u_{j'}}, p)$
is a column defective noncrossing if and only if
$(\hat \pi_i, \hat \pi_{u_{j'}}, p)$ is.
Finally, observe that neither triple in
(\ref{eq:piuj}) can be a column defective noncrossing if
$(\pi_{u_j}, \pi_{u_{j'}}, p)$ is proper,
while exactly one of the two a column defective noncrossing if
$(\pi_{u_j}, \pi_{u_{j'}}, p)$ is defective.
Thus the points of intersection of the forms (\ref{eq:piuj})
contribute a surplus of
$1$ to $\cdncross(\pi)$
if $(\pi_{u_j}, \pi_{u_{j'}}, p)$ is a defective noncrossing,
$1$ to $\cdncross(\hat \pi)$
if $(\pi_{u_j}, \pi_{u_{j'}}, p)$ is a defective crossing,
and $0$ to both otherwise.
\end{proof}

As a consequence, we have that the map $\zeta$ preserves
a certain linear combination of the above statistics.
\begin{cor}\label{c:adamidentity}
  For $W = \delta(U(\pi, u, yu))$ and $\zeta(W) = \delta(U(\hat \pi, u, \hat yu))$ in $\mathcal T_I$, we have
  \begin{equation}\label{eq:adamidentity}
    \frac{\cross(\hat \pi)}2 + \cdncross(\zeta(W)) + \incross(\zeta(W))
    - \frac{\ell(\hat y)}2
    = \frac{\cross(\pi)}2 + \cdncross(W) + \incross(W) - \frac{\ell(y)}2.
  \end{equation}
  \end{cor}
\begin{proof}
  If $W$ is a fixed point of $\zeta$, then the result is clear.
  Suppose therefore that $\zeta(W) \neq W$, and consider the triple
  $(\pi_{u_j}, \pi_{u_{j'}}, p)$ appearing in the definition of $\zeta$.
  If this triple is a proper noncrossing, then by Lemma~\ref{l:statids}
  the left-hand side of
  (\ref{eq:adamidentity}) is
    \begin{multline*}
    \frac{\cross(\pi) + 1 + 2b}2 + \cdncross(W) + \incross(W) - b
    - \frac{\ell(y)+1}2\\
    = \frac{\cross(\pi)}2 + \cdncross(W) + \incross(W) - \frac{\ell(y)}2
    + \frac{1 + 2b}2 - b - \frac 12.
    \end{multline*}
  One shows similarly that the result holds when the triple is a
  defective noncrossing or any crossing.
  \end{proof}

Finally we can state and justify a
subtraction-free formula for
$\epsilon_q^\lambda(\wtc{s_{J_1}\ntnsp}q \cdots \wtc{s_{J_m}\ntnsp}q)$.
The proof of the following result is nearly identical to that of
\cite[Thm.\,5.4]{KLSBasesQMBIndSgn}.
We include it here, both for the convenience of the reader,
and because it relies upon
Theorem~\ref{t:qstem},
Proposition~\ref{p:sigmastats},
Lemmas~\ref{l:TU} -- \ref{l:lengthdiff}
in this paper which differ from
the corresponding (weaker) results
\cite[Thm.\,3.7]{KLSBasesQMBIndSgn},
\cite[Prop.\,4.1]{KLSBasesQMBIndSgn},
\cite[Lem.\,5.1 -- 5.2]{KLSBasesQMBIndSgn},
in that paper,
and upon Corollary~\ref{c:adamidentity}.

\begin{thm}\label{t:qepsilon}
Let $G = G_{J_1} \circ \cdots \circ G_{J_m}$.
Then for $\lambda \vdash n$ 
we have
\begin{equation}\label{eq:epsilonmain}
  \epsilon_q^\lambda(\wtc{s_{J_1}\ntnsp}q \cdots \wtc{s_{J_m}\ntnsp}q)  
  =
\sum_\pi \qp{\cross(\pi)}2 \sum_W q^{\incross(W)},
\end{equation}
where the sums are over path families $\pi$ of type $e$ which cover $G$,
and 
column-strict $\pi$-tableaux $W$
of 
shape $\lambda^\tr$.
\end{thm}
\begin{proof}
  Let $B$ be the path matrix of $G$.
  By Theorem~\ref{t:qstem}
  and
  Equations (\ref{eq:immepsilon}) -- (\ref{eq:lambepsilon}),
  the left-hand side of (\ref{eq:epsilonmain}) is
  \begin{equation}\label{eq:epsilonsecond}
    \begin{aligned}
      \sigma_B(\imm{\epsilon_q^\lambda}(x)) &=
      \sum_I
      \sigma_B(\qdet(x_{I_1,I_1}) \cdots \qdet(x_{I_r,I_r}))\\
      &=
      \sum_I
      \sum_{\smash{y \in \slambda}} (-1)^{\ell(y)}\qiey
      \sigma_B (x^{u(I),yu(I)}),
    \end{aligned}
  \end{equation}
    where the first two sums are over ordered set partitions
    $I = (I_1,\dotsc,I_r)$ of $[n]$ of type $\lambda$.
    Fixing one such partition $I$ and writing $u = u(I)$, we may use
    Proposition~\ref{p:sigmastats} and Lemma~\ref{l:TU} to
    express the sum over $y \in \slambda$ as
\begin{equation}\label{eq:ypiinvuupi}
  \sum_{y \in \slambda}
  \sum_\pi 
  (-1)^{\ell(y)} \qiey
\qp{\cross(\pi)}2
q^{\incross(U(\pi,u,yu))}
=
\sum_{y \in \slambda}
  \sum_\pi 
(-1)^{\ell(y)} \qiey
\qp{\cross(\pi)}2 q^{\incross(W)+\cdncross(W)},
\end{equation}
where
the inner
sums are over
path families $\pi$ of type
$u^{-1}yu$
which cover $G$,
and where $W = \delta_I(U(\pi,u,yu))$.
As $y$ and $\pi$ vary in the above sums,
$U(\pi,u,yu)$
varies over all tableaux in $\mathcal U_I$,
and
$W$ varies over all tableaux in $\mathcal T_I$. 

Now consider a tableau $W \in \mathcal T_I$ which satisfies $\zeta(W) \neq W$.
Let tableaux $W$ and $\zeta(W)$ contain path families
$\pi$ of type $u^{-1}yu$ and $\hat \pi$ of type $u^{-1}\hat yu$, respectively.
Then the two terms
on the right-hand side of
(\ref{eq:ypiinvuupi}) corresponding to
$\zeta(W)$ and $W$ sum to
\begin{equation*}
  (-1)^{\ell(\hat y)}
  \qumongous{\cross(\hat \pi)}{\incross(\zeta(W))}{\cdncross(\zeta(W))}{\ell(\hat y)}
  + (-1)^{\ell(y)}
  \qumongous{\cross(\pi)}{\incross(W)}{\cdncross(W)}{\ell(y)}.
  \end{equation*}
By Lemma~\ref{l:lengthdiff}
and Corollary~\ref{c:adamidentity}
this is $0$.
Thus it suffices to sum the right-hand side of (\ref{eq:ypiinvuupi})
over only the pairs $(y,\pi)$ corresponding
to tableaux $W$ satisfying $W = \zeta(W)$.
By the definition of $\zeta$, each such tableau $W$
is column-strict and therefore satisfies $\cdncross(W) = 0$.
Since all tableaux in $\mathcal T_I$ are also column-closed,
each such tableau $W$ must have type $e$.  Thus we have
we have $u^{-1}yu = e$, i.e., $y=e$.
It follows that the
right-hand side of (\ref{eq:ypiinvuupi})
and the third sum in (\ref{eq:epsilonsecond}) are equal to
\begin{equation*}
 \qp{\cross(\pi)}2 \sum_W  q^{\incross(W)},
\end{equation*}
where the sum is over all tableau $W$ in $\mathcal T_I$ which
are column-strict of type $e$.
Since all tableaux in $\mathcal T_I$ have shape $\lambda^\tr$,
the three expressions in (\ref{eq:epsilonsecond})
are equal to the right-hand side of (\ref{eq:epsilonmain}).
\end{proof}

To illustrate the theorem, we compute 
$\epsilon_q^{211}(\wtc{s_{[1,2]}\ntnsp}q \wtc{s_{[2,4]}\ntnsp}q \wtc{s_{[1,2]}\ntnsp}q)$ 
using the wiring diagram
\begin{equation}
G = G_{[1,2]}\circ G_{[2,4]} \circ G_{[1,2]} = 
\begin{tikzpicture}[scale=.5,baseline=-10]
  \draw (0,1) -- (1,1) -- (1.5,0) -- (2,1) -- (3,1);
  \draw (0,0) -- (1,0) -- (1.5,0) -- (2,0) -- (3,0);
  \draw (0,-1) -- (.5,-1.5) --  (1,-1) -- (1.5,0) -- (2,-1) -- (2.5,-1.5) -- (3,-1);
  \draw (0,-2) -- (.5,-1.5) -- (1,-2) -- (2,-2) -- (2.5,-1.5) -- (3,-2);
\end{tikzpicture}.
\end{equation}

There are two path families of type $e$ which cover $G$, and four 
column-strict $G$-tableaux of shape $211^\tr = 31$ for each:
\begin{equation*}\label{eq:Gtableauxpi}
\begin{tikzpicture}[scale=.5,baseline=-10]
  \draw[dashed, ultra thick] (0,1) -- (1,1) -- (1.5,0) -- (2,1) -- (3,1);
  \draw[dotted, thick] (0,0) -- (1,0) -- (1.5,0) -- (2,0) -- (3,0);
  \draw[dashed] (0,-1) -- (.5,-1.5) --  (1,-1) -- (1.5,0) -- (2,-1) -- (2.5,-1.5) -- (3,-1);
  \draw[-,thick] (0,-2) -- (.5,-1.5) -- (1,-2) -- (2,-2) -- (2.5,-1.5) -- (3,-2);
  \node at (-.5,1) {$\pi_4$};
  \node at (-.5,0) {$\pi_3$};
  \node at (-.5,-1) {$\pi_2$};
  \node at (-.5,-2) {$\pi_1$};
\end{tikzpicture}
\;, \quad
U_\pi^{(1)} = \tableau[scY]{\pi_3 | \pi_1,\pi_2,\pi_4} \;,\ 
U_\pi^{(2)} = \tableau[scY]{\pi_3 | \pi_1,\pi_4,\pi_2}\;,\ 
U_\pi^{(3)} = \tableau[scY]{\pi_4 | \pi_1,\pi_2,\pi_3}\;,\ 
U_\pi^{(4)} = \tableau[scY]{\pi_4 | \pi_1,\pi_3,\pi_2}\;; 
\end{equation*}
\begin{equation*}
    \begin{tikzpicture}[scale=.5,baseline=-10]
    \draw[dashed, ultra thick] (0,1) -- (1,1) -- (1.5,0) -- (2,1) -- (3,1);
    \draw[dotted, thick] (0,0) -- (1,0) -- (1.5,0) -- (2,0) -- (3,0);
    \draw[dashed] (0,-1) -- (.5,-1.5) -- (1,-2) -- (2,-2) -- (2.5,-1.5) -- (3,-1);
    \draw[-, thick] (0,-2) -- (.5,-1.5)  --  (1,-1) -- (1.5,0) -- (2,-1) -- (2.5,-1.5) -- (3,-2);
    \node at (-.5,1) {$\rho_4$};
    \node at (-.5,0) {$\rho_3$};
    \node at (-.5,-1) {$\rho_2$};
    \node at (-.5,-2) {$\rho_1$};
    \end{tikzpicture}\;, \quad
    U_\rho^{(1)} = \tableau[scY]{\rho_3 | \rho_2,\rho_1,\rho_4}\;,\ 
    U_\rho^{(2)} = \tableau[scY]{\rho_3 | \rho_2,\rho_4,\rho_1}\;,\ 
    U_\rho^{(3)} = \tableau[scY]{\rho_4 | \rho_2,\rho_1,\rho_3}\;,\ 
    U_\rho^{(4)} = \tableau[scY]{\rho_4 | \rho_2,\rho_3,\rho_1}\;.
    \end{equation*}
Since the path family $\pi$ has no crossings, 
we have $\textsc{c}(U_\pi^{(i)}) = 0$ for all $i$, and each
tableau $U_\pi^{(i)}$ therefore contributes 
$q^{\textsc{invnc}(U_\pi^{(i)})}$. 
We have one noncrossing for each of the pairs $(\pi_2,\pi_3)$, $(\pi_2,\pi_4)$ and $(\pi_3,\pi_4)$ 
and two for the pair $(\pi_1,\pi_2)$. Counting the noncrossings only for pairs where the path 
which intersects the other from above appears in a column left of the other, for instance $\pi_2$ and $\pi_3$ 
in $U_\pi^{(1)}$, we find the contributions from $U_\pi^{(1)},\dots,U_\pi^{(4)}$ are
$q, q^2, q^2, q^3$, respectively. 
Since the path family $\rho$ has two crossings, 
the tableaux for the path family $\rho$ each have two crossings, 
and one noncrossing for 
each of the pairs $(\rho_1,\rho_3)$, $(\rho_1,\rho_4)$ and $(\rho_3,\rho_4)$. 
Adding the contributions together we find the contributions for 
$U_\rho^{(1)},\dots,U_\rho^{(4)}$ are $q^1q^{2/2} = q^2$, $q^2q^{2/2} = q^3$, 
$q^2q^{2/2} = q^3$ and $q^3q^{2/2}= q^4$ respectively. 
Hence we have 
$\epsilon_q^{211}(\wtc{s_{[1,2]}\ntnsp}q \wtc{s_{[2,4]}\ntnsp}q \wtc{s_{[1,2]}\ntnsp}q)
= q + 3q^2 + 3q^3 + q^4$.

Theorem~\ref{t:qepsilon} allows one to
combinatorially interpret evaluations of $\epsilon_q^\lambda$ at
(multiples of) certain elements $\wtc wq$ 
of the Kazhdan-Lusztig basis of $\hnq$.
In particular, for some elements $\wtc wq$ there exists a polynomial $g(q)$
such that we have
  \begin{equation}\label{eq:klfactor}
    g(q) \wtc wq = \wtc{s_{J_1}\ntnsp}q \cdots \wtc{s_{J_m}\ntnsp}q
  \end{equation}
  for some sequence $s_{J_1},\dotsc,s_{J_m}$ of reversals.
  Such permutations include all \pavoiding permutations,
  all of $\mfs 4$ (even $4231$ and $3412$),
  all of $\mfs 5$ except $45312$,
  and all $321$-{\em hexagon-avoiding} permutations.
  (See \cite{BWHex}.)
  
\begin{cor}\label{c:klfactor}
  Suppose that $\wtc wq$ satisfies
  a factorization of the form (\ref{eq:klfactor})
  and define $G = G_{J_1} \circ \cdots \circ G_{J_m}$.
  Then we have
  \begin{equation}\label{eq:klfactoreval}
    \epsilon_q^\lambda(\wtc wq) = \frac1{g(q)}\sum_U q^{\incross(U)+\cross(U)/2},
\end{equation}
where the sum is over all column-strict $G$-tableaux 
of type $e$ and shape $\lambda^\tr$.
\end{cor}

It would be interesting to know for which Kazhdan-Lusztig basis
elements we have the above factorization~(\cite[Quest.\,4.5]{SkanNNDCB}).

Recall from Equations~(\ref{eq:cprodq}) -- (\ref{eq:cnoprodq})
that Haiman's result~\cite[Lem.\,1.1]{HaimanHecke} that we have
\begin{equation*}
  \chi_q^\lambda(\wtc wq) \in \mathbb N[q]
  \end{equation*}
for all irreducible characters $\chi_q^\lambda$ and all $w \in \sn$
implies that we also have
\begin{equation}\label{eq:prodchi}
  \chi_q^\lambda(\wtc{s_{J_1}\ntnsp}q \cdots \wtc{s_{J_m}\ntnsp}q) \in \mathbb N[q]
\end{equation}
for all irreducible characters $\chi_q^\lambda$ and
all sequences $(s_{J_1}, \dotsc, s_{J_m})$ of reversals.
It would therefore be interesting to extend Theorem~\ref{t:qepsilon}
to irreducible characters.

\begin{prob}\label{p:qchi}
  Find a combinatorial interpretation of the polynomial
  (\ref{eq:prodchi})
  which holds for all irreducible charcters $\chi_q^\lambda$
  and all sequences $(s_{J_1}, \dotsc, s_{J_m})$ of reversals.
\end{prob}



\end{document}